\documentclass[10pt, english]{amsart}

\usepackage{amsmath,amssymb,enumerate}

\usepackage[T1]{fontenc}
\usepackage[all,cmtip]{xy}

\usepackage{babel}
\usepackage{amstext}
\usepackage{amsmath}
\usepackage{amsfonts}
\usepackage{latexsym}
\usepackage{ifthen}

\usepackage{xypic}
\xyoption{all}
\newcommand{\rk}{{\rm rk}}
\newcommand{\codim}{{\rm codim}}

\newtheorem{lemma1}{}[section]

\newenvironment{lemma}{\begin{lemma1}{\bf Lemma.}}{\end{lemma1}}
\newenvironment{example}{\begin{lemma1}{\bf Example.}\rm}{\end{lemma1}}

\newenvironment{theorem}{\begin{lemma1}{\bf Theorem.}}{\end{lemma1}}

\newenvironment{proposition}{\begin{lemma1}{\bf Proposition.}}{\end{lemma1}}
\newenvironment{corollary}{\begin{lemma1}{\bf Corollary.}}{\end{lemma1}}
\newenvironment{remark}{\begin{lemma1}{\bf Remark.}\rm}{\end{lemma1}}

\newenvironment{definition}{\begin{lemma1}{\bf Definition.}}{\end{lemma1}}
\newenvironment{notation}{\begin{lemma1}{\bf Notation.}}{\end{lemma1}}

\newenvironment{setup}{\begin{lemma1}{\bf Setup.}}{\end{lemma1}}

\newenvironment{fact}{\begin{lemma1}{\bf Fact.} \rm}{\end{lemma1}}

\newenvironment{the local obstruction - setup}{\begin{lemma1}{\bf The local obstruction - setup.}}{\end{lemma1}}

\newenvironment{remark*}{{\bf Remark.}}{}
\newenvironment{remarks*}{{\bf Remarks.}}{}
\newenvironment{example*}{{\bf Example.}}{}
\newenvironment{assumption*}{{\bf Assumption.}}{}

\newcommand{\R}{\ensuremath{\mathbb{R}}}
\newcommand{\Q}{\ensuremath{\mathbb{Q}}}
\newcommand{\Z}{\ensuremath{\mathbb{Z}}}
\newcommand{\C}{\ensuremath{\mathbb{C}}}
\newcommand{\N}{\ensuremath{\mathbb{N}}}
\newcommand{\PP}{\ensuremath{\mathbb{P}}}

\usepackage{eurosym}

\newcommand{\merom}[3]{\ensuremath{#1:#2 \dashrightarrow #3}}

\newcommand{\holom}[3]{\ensuremath{#1:#2  \rightarrow #3}}
\newcommand{\fibre}[2]{\ensuremath{#1^{-1} (#2)}}

\makeatletter
\ifnum\@ptsize=0  \fi \ifnum\@ptsize=2  \fi 

\newcommand\sE{{\mathcal E}}

\newcommand\sF{{\mathcal F}}

\newcommand\sL{{\mathcal L}}

\newcommand\sO{{\mathcal O}}

\newcommand\sS{{\mathcal S}}

\DeclareMathOperator*{\pic}{Pic}
\DeclareMathOperator*{\sing}{sing}

\DeclareMathOperator*{\Supp}{Supp}

\DeclareMathOperator*{\supp}{Supp}

\newcommand{\NE}[1]{ \ensuremath{ \overline { \mbox{NE} }(#1)} }

\usepackage{xcolor}
\usepackage{hyperref}

\author{Andreas H\"oring}
\author{Thomas Peternell}

\address{Andreas H\"oring, Universit\'e C\^ote d'Azur, CNRS, LJAD, France, Institut universitaire de France}
\email{Andreas.Hoering@univ-cotedazur.fr}

\address{Thomas Peternell, Mathematisches Institut, Universit\"at Bayreuth, 95440 Bayreuth, 
Germany}
\email{thomas.peternell@uni-bayreuth.de}

\subjclass[2010]{14D06, 14J45, 14E05, 14E30}
\keywords{}

\title{Klt degenerations of projective spaces } 
\date{\today} 

\begin{document}

\begin{abstract} 
We study degenerations of complex projective spaces $\PP^n$ into normal projective klt varieties $X$. If the tangent sheaf of $X$ is
semi-stable, we show that $X$ itself is a projective space. 
If $X$ is a threefold with canonical singularities, we show that there are only three possible families of varieties with this property. 
\end{abstract} 
%\tableofcontents

\maketitle

\section{Introduction}

\subsection{Motivation and main results}

A famous theorem of Hirzebruch and Kodaira, \cite{HK57} states that a compact K\"ahler manifold  $X$ which is diffeomorphic to $\mathbb P^n$
is actually biholomorpic to $\mathbb P^n$, unless possibly  $n$ is even  and $K_X$ is ample. This last case is ruled out by Yau's theorem \cite{Yau77}; the existence of a
K\"ahler-Einstein metric.  As a corollary, one obtains the deformation invariance of $\mathbb P^n$:

\begin{theorem}  \label{thm:HK} Let $\pi: \mathcal X \to \Delta $ be a proper projective  submersion to the  unit disc in $\mathbb C$, such that 
$\mathcal X_t = \pi^{-1}(t) \simeq \mathbb P^n$ for $t \ne 0$. Then $\mathcal X_0 \simeq \mathbb P^n$. 
\end{theorem} 

In fact, by Ehresmann's theorem, $\mathcal X _0 $ is diffeomorphic to $\mathbb P^n$. 
In Theorem \ref{thm:HK}, it is actually not necessary to assume that $X$ is projective, as shown by Siu \cite{Siu89,Siu92}, see also \cite{Hwa96}. 

In this paper we study singular degenerations of $\mathbb P^n$ (see Definition \ref{def1} for the setup), namely we suppose that the central fiber is a normal projective variety with a restricted type of singularities, 
so-called klt varieties, see e.g. \cite{KM98}, which play an important role in the minimal model theory. 
In this setting, the central fiber $X_0$ is of course no longer diffeomorphic to $\mathbb P^n$, and one cannot expect $X_0 $ to be a projective space (cf. Example \ref{example-degeneration-cone}). Degenerations of $\PP^2$  have been studied and classified by 
Manetti \cite{Man91} and in a more general context by Hacking-Prokhorov \cite{HP10}. 
 
Despite of the existence of many singular degenerations of projective space,  we have the following rigidity result. 
 
 \begin{theorem} \label{theorem-stability}
 Let $X$ be a normal projective klt degeneration of $\PP^n$. If $T_X$ is semistable, then $X \simeq \PP^n$.
\end{theorem}

The variety $X$ has Picard number one by Propsition \ref{proposition-picard-number},
so stability is uniquely defined in terms of the polarisation.
Our technique can be adapted to degenerations of other Fano manifolds with Picard number one,
e.g. it is not very to prove the analogue of Theorem \ref{theorem-stability}
for degenerations of smooth quadrics $Q_n$ of dimension $n \geq 3$.

In case $T_X$ is not semistable, the task is to classify $X$. We will restrict ourselves to degenerations which have canonical singularities and prove

\begin{theorem} \label{theorem-dim3}
Let $X$ be a normal projective canonical degeneration of $\PP^3$.
Then one of the following holds:
\begin{enumerate}
\item $X \simeq \PP^3$; or
\item $X$ is the contraction of the section $\PP(\sO_S)$ in $\PP(\sO_S \oplus \sO_S(-K_S))$  where $S$ is a quadric surface; or
\item $X$ is the variety in Example \ref{the-example}. 
\end{enumerate}
\end{theorem} 

The variety in Example \ref{the-example} has quite remarkable properties: one has
$-K_X \simeq 4 H$ with $H$ an ample Weil $\Q$-Cartier divisor and
$$
H \simeq 3 F + B
$$
where $B$ is the fixed part of $|H|$ and $F$ is a pencil of del Pezzo surfaces of degree four
with a $D_5$-singularity. The divisors $F$ and $B$ are not $\Q$-Cartier, but $F$
defines a fibration on a small modification of $X$. The variety $X$ has the same Hilbert series and invariants as a projective space, e.g.
\begin{multline*}
h^0(X,T_X) = h^0(\PP^3, T_{\PP^3}),
\ h^0(X,T_X(-1)) = h^0(\PP^3, T_{\PP^3}(-1)), \\ 
h^0(X, (\bigwedge^2 T_X)^{**}) = h^0(\PP^3, \bigwedge^2 T_{\PP^3}), \
h^0(X, (\bigwedge^2 T_X(-1))^{**}) = h^0(\PP^3, \bigwedge^2 T_{\PP^3}(-1)).
\end{multline*}

After this paper was finished we learned that Ascher-de Vleming-Liu \cite[Thm.1.5]{ADL23}
used $K$-stability and the theory of K3 surfaces to give a classification of canonical Gorenstein degenerations of $\PP^3$, obtaining the same list as in Theorem \ref{theorem-dim3}. In particular, they prove that the variety in 
 Example \ref{the-example} is indeed a degeneration of $\PP^3$ 
Our technical approach is completely different, thereby providing an additional tool for the vast theory of moduli spaces of $K$-stable varieties.

As an immediate consequence of Theorem \ref{theorem-dim3} we obtain:
 
\begin{corollary} Let $X$ be a normal projective terminal degeneration of $\PP^3$. Then we have
$X \simeq \PP^3$.
\end{corollary}
 
\subsection{Outline of the paper}
Let 
$$
\holom{\pi}{\mathcal X}{\Delta}
$$
be a projective morphism with connected fibres onto a disc $\Delta$ (cf. \cite[Thm.5]{AR14} for degenerations over a polydisc) such that
$\mathcal X_t \simeq \PP^n$ for $t \neq 0$, and denote by $A \rightarrow \mathcal X$ a
relatively ample Cartier divisor. If we can choose $A$ such that $\sO_{\mathcal X_t}(A) \simeq
\sO_{\PP^n}(1)$ on a general fibre, a result of Araujo and Druel \cite[Prop.4.10]{AD14} 
implies that $\mathcal X \rightarrow \Delta$ is a projective bundle.
In our Setup \ref{setup-degen}, with $X$ being klt,  the degeneration of the hyperplane class merely defines a 
Weil $\Q$-Cartier divisor $H$ on the central fibre $X := \mathcal X_0$ such that
$$
-K_X \simeq (n+1) H, \qquad
h^0(X, \sO_X(H)) \geq n+1.
$$
 If $X \not\simeq \PP^n$, the base locus of the linear system $|H|$ is not empty, so we obtain a rational map
$$
\merom{\varphi_{|H|}}{X}{\PP^N}
$$
which is the basis of our investigation. We start by showing that if $\varphi_{|H|}$ is generically finite onto its image $Z \subset \PP^N$, then $X$ is isomorphic to the projective space. If $\varphi_{|H|}$ is not generically finite its fibres define a foliation $\sF \subset T_X$. For the proof of Theorem \ref{theorem-stability}  we show that this subsheaf destabilises the tangent sheaf.

In view of the numerous possibilities for klt degenerations even in dimension two,
it seems reasonable to restrict the classification in higher dimension 
to varieties with canonical singularities. As a first step we establish in Theorem \ref{theorem-case-nminusone} an upper bound on the degree of $Z$ if the general $\varphi_{|H|}$-fibre
is a curve. This bound is optimal and indicates that this case will always be realised by
the cone construction from Example \ref{example-degeneration-cone}.
With this preparation the remainder of the classification of canonical degenerations of $\PP^3$ reduces to the discussion of the apparently innocuous fibrations
$$
\merom{\varphi_{|H|}}{X}{\PP^1 \subset \PP^N}.
$$
If the fixed part of $|H|$ is empty, the anticanonical divisor is divisible by
$4 (h^0(X, \sO_X(H))-1) \geq 12$ in the class group, a very strong restriction on the geometry
of $X$, cp. Prokhorov's work \cite{Pro10} in the terminal case. Unfortunately, as shown by Example \ref{the-example}, we can not make this assumption if $X$ is not $\Q$-factorial. However we show in Proposition \ref{proposition-classify-Xprime} that there exists a $\Q$-factorialisation $X' \rightarrow X$
where the mobile part is basepoint-free and defines a fibration
$X' \rightarrow \PP^1$
such that the general fibre is a del Pezzo surface of degree four having a $D_5$-singularity.
We then consider the rational map $\varphi_{|2H|}$: it turns out that the graph of this map has a conic bundle structure over a surface of minimal degree in $\PP^9$, this finally leads
to the classification in Theorem \ref{theorem-dim3}.

{\bf Acknowledgements.} 
 A.H.\ thanks the Institut Universitaire de France for providing excellent working conditions for this project. A crucial part of Section \ref{section-dim3} was done during his stay at IBS-CCG, he would like to thank J.-M. Hwang and Y. Lee for their hospitality. T.P. recognizes support by the German Research Council (DFG;  project Pe305/14-1). 
We thank Yuchen Liu for bringing \cite{ADL23} to our attention.

\section{Notation and basic definitions}

We work over the complex numbers, for general definitions we refer to \cite{Har77, Kau83}. 
Complex spaces and varieties will always be supposed to be irreducible and reduced. 
We follow \cite{Laz04} for algebraic notions of positivity.
For positivity notions in the analytic setting, we refer to \cite{Dem12}.

We use the terminology of \cite{Deb01, KM98}  for birational geometry and notions from the minimal model program. We denote by $\simeq$ ($\sim_\Q$) the linear equivalence of Weil (resp. $\Q$-Weil) divisors
on a normal variety.

Given a normal variety $X$, a Zariski open subset $X^0 \subset X$ is big if the complement
$X \setminus X^0$ has codimension at least two.

We will frequently use the following:

\begin{notation} {\rm \label{notation-Qfactorial}
Let $X$ be normal projective variety with klt (resp. canonical, resp. terminal) singularities. Then by \cite[Cor.1.37]{Kol13} there exists a $\Q$-factorialisation
$$
\holom{\eta}{X'}{X},
$$
i.e., a {\it small} birational morphism such that  $X'$ is $\Q$-factorial with klt (resp. canonical, resp. terminal) singularities. Let $D$ be a Weil $\Q$-Cartier divisor on $X$. Then the pull-back
$\eta^* D$ is well-defined as a $\Q$-Cartier divisor class. Since $\eta$ is small, the divisor $\eta^* D$ coincides with the strict transform $D' \subset X'$ in codimension one,
so $\eta^* D=\bar D$ is a {\em Weil} $\Q$-Cartier divisor on $X'$. In particular there is
an associated reflexive sheaf $\sO_{X'}(\eta^* D)$ and $\eta$ being small we have
$$
H^0(X', \sO_{X'}(\eta^* D)) \simeq H^0(X, \sO_{X}(D)). 
$$
If 
$$
D \simeq M+B
$$
is the decomposition of the linear system $|D|$ in
its mobile and fixed part, then 
$$
\eta^* D \simeq D' = M' + B'
$$ 
is the decomposition into
mobile and fixed part of $|D'|$. 

Note that in general $M$ and $B$ are Weil, but not $\Q$-Cartier, so $\eta^* M$ and $\eta^* B$ are not well-defined.}
\end{notation}

\begin{remark} \label{remark-rather-big}
In the situation above, assume that $D$ is ample. Then 
the $\Q$-Cartier divisor $D'$ is nef and big and semiample, but not ample unless $X$ is $\Q$-factorial. However, $\eta$ being small, the restriction of $D'$ to any prime divisor $E \subset X'$ is big, i.e. thus we have
$$
D'^{\dim X-1} \cdot E>0.
$$
\end{remark}

\begin{definition}
Given a normal projective variety $X$ we denote by $N^d(X)$ the $\R$-vector space of cycles
of codimension $d$ modulo numerical equivalence. Given two elements $A, B \in N^d(X)$
we will write
$$
A \geq B
$$
if the cycle class $A-B$ is in the pseudoeffective cone $\mbox{\rm Pseff}^d(X) \subset N^d(X)$, i.e., $A-B$ 
is a limit of effective cycle classes.
\end{definition}

\begin{fact} \label{fact-subadjunction}
Let $X$ be a normal $\Q$-factorial variety. Let $E \subset X$ be a prime divisor, and let $\holom{\nu}{\tilde E}{E}$
be the normalisation. Then by \cite[Ch.3]{Sho92} there exists a canonically defined effective $\Q$-divisor $\Delta$ on $\tilde E$ such that the subadjunction formula
\begin{equation}
\label{eqn-subadjunction}
K_{\tilde E} + \Delta \sim_\Q \nu^* (K_X+E)
\end{equation}
holds. In particular we have
$$
- K_{\tilde E} \geq - \nu^* (K_X+E).
$$
\end{fact}

\begin{remark} \label{non-Q-factorial}
Let $\bar X$ be a normal variety of dimension $n$ that is $\Q$-factorial in codimension two,
and let $N \subset \bar X$ be the subset of codimension at least three such that $\bar X \setminus N$ is $\Q$-factorial.
Let $\holom{\nu}{\tilde B}{X}$ be a finite morphism from a normal variety of dimension $n-1$,
and denote by $B \subset X$ its image. 
Then we have a well-defined pull-back
$$
\mbox{Cl}(\bar X \setminus N) \otimes \Q \simeq \pic(\bar X \setminus N) \otimes \Q \rightarrow \pic(\tilde B \setminus \fibre{\nu}{N}) \otimes \Q \simeq \mbox{Cl}(\tilde B \setminus \fibre{\nu}{N}) \otimes \Q.
$$
Since $\fibre{\nu}{N}$ has codimension at least two in $\tilde B$, we can extend to
$$
\nu^*: \mbox{Cl}(\bar X) \otimes \Q \rightarrow \mbox{Cl}(\tilde B) \otimes \Q.
$$
In particular the subadjunction formula \eqref{eqn-subadjunction} is still valid
for the Weil divisor $B \subset X$.
\end{remark}

The following is of course well-known in case $D$ is Cartier.

\begin{lemma} \label{lemma-indeterminacies}
Let $X$ be a normal projective variety and let $D$ be a Weil $\Q$-Cartier divisor on $X$.
Assume that $h^0(X, \sO_X(D)) \geq 2$ and the linear system $|D|$ has no fixed components.
Then there exists a birational morphism from a normal projective variety
$$
\holom{\mu}{Y}{X}
$$
such that 
$$
\mu^* D \sim_\Q M_Y + \sum_{i=1}^k a_i E_i
$$
where
\begin{itemize}
\item $\sum_{i=1}^k a_i E_i$ is an effective $\Q$-Weil divisor such that $\Supp(\sum_{i=1}^k a_i E_i)=\mbox{Exc}(\mu)$; and
\item $M_Y$ is a globally generated $\mu$-ample Cartier divisor such that $h^0(X, \sO_X(D))= h^0(Y, \sO_Y(M_Y))$ and
$$
\holom{\varphi_{|M_Y|}}{Y}{Z}
$$
is a resolution of the indeterminacies of $\varphi_{|D|}$. We call $Z$ the (closure of the) image of $\varphi_{|D|}$.
\end{itemize}
\end{lemma}

\begin{proof} Set $r:= h^0(X, \sO_X(D))$ and
let $\sS \subset \sO_X(D)$ be the image of the evaluation morphism
$$
\sO_X^{\oplus r} \rightarrow \sO_X(D).
$$
Since $|D|$ has no fixed component, the torsion-free sheaf $\sS$ coincides with $\sO_X(D)$ in codimension one. Let $\PP(\sS) \subset \PP(\sO_X^{\oplus r})$
be the projectivisation, and let $\holom{\nu}{Y}{\PP(\sS)}$ be the normalisation
of the unique irreducible component of $\PP(\sS)$ that dominates $X$.
Set $M_Y:=\nu^* c_1(\sO_{\PP(\sS)}(1))$, then by construction $M_Y$ is Cartier and globally generated, in particular it is nef. Since $\sO_{\PP(\sS)}(1)$ is relatively ample and $\nu$ is finite,
the divisor $M_Y$ is $\mu$-ample

Consider now the birational morphism $\holom{\mu}{Y}{X}$, and let 
$E_1, \ldots, E_k$ be the $\mu$-exceptional divisors.
Since $\sO_{\PP(\sS)}(1)$
coincides with $\sO_X(D)$ in the locus where $\sO_X(D)$ is locally free and globally generated,
the Cartier divisor $M_Y$ coincides with the $\Q$-Cartier divisor $\mu^* D$ in the complement of the exceptional locus. Thus 
$$
M_Y-\mu^* D \sim_\Q - \sum_{i=1}^k a_i E_i
$$
has support in the exceptional locus. Since $M_Y$ is nef, the divisor $M_Y-\mu^* D$ is $\mu$-nef.
Thus we can apply the negativity lemma \cite[Lemma 3.37]{KM98} to see that 
$-(M_Y-\mu^* D)$ is effective, i.e. we have $a_i \geq 0$ for all $i=1, \ldots, k$. 

We claim that $\Supp(\sum_{i=1}^k a_i E_i)=\mbox{Exc}(\mu)$, hence in particular $a_i>0$ for all $i=1, \ldots, k$. Arguing by contradiction assume that the
inclusion $\Supp(\sum_{i=1}^k a_i E_i) \subset \mbox{Exc}(\mu)$ is strict, so there exists an irreducible component $Z$ of $\mbox{Exc}(\mu)$ that is not in $\Supp(\sum_{i=1}^k a_i E_i)$. In particular if $F$
is a general fibre of $Z \rightarrow \mu(Z)$, it is not contained in  $\Supp(\sum_{i=1}^k a_i E_i)$.
Thus the restriction 
$$
\sum_{i=1}^k a_i E_i|_F \sim_\Q -M_Y|_F
$$ 
is effective and antiample, a contradiction.

Finally let us see that $h^0(Y, \sO_Y(M_Y))=r$. By construction
$h^0(\PP(\sS), \sO_{\PP(\sS)}(1))=r$, moreover the pull-back
$$
H^0(X, \sO_X(D)) \simeq 
H^0(\PP(\sS), \sO_{\PP(\sS)}(1)) \rightarrow H^0(Y, \sO_Y(M_Y))
$$
is injective: if a section $s \in H^0(\PP(\sS), \sO_{\PP(\sS)}(1))$ vanishes on $Y$, it vanishes on a Zariski open subset of $X$. Yet $\sO_X(D)$ is torsion-free, so $s=0$.
Thus we have $h^0(Y, \sO_Y(M_Y)) \geq r$. Since the torsion-free sheaf $\mu_* \sO_Y(M_Y)$ injects into its bidual $\sO_X(D)$,
we have equality.
\end{proof}

\section{$\Delta$-genus for Weil divisors}

The $\Delta$-genus of a nef and big Cartier divisor was defined by Fujita \cite{Fuj89},
for our considerations we have to extend parts of his theory to Weil $\Q$-Cartier divisors.

\begin{definition} \label{definition-deltagenus}
Let $X$ be a normal projective variety of dimension $n$, and let $D$ be a Weil $\Q$-Cartier divisor that is nef and big. The $\Delta$-genus of $(X,D)$ is defined as
$$
\Delta(X,D) = n + D^n - h^0(X, \sO_X(D)) 
$$
\end{definition}

\begin{theorem} \label{thm:delta-0} Let $X$ be a normal projective klt variety of dimension $n$. Let $D$ be a Weil $\mathbb Q$-Cartier  divisor on $X$ that is nef and big. Assume that
$$ 
h^0(X,\sO_X(D)) \geq n + D^n.
$$ 
and the rational morphism given by $\vert D \vert$ is generically finite on its image.
Then $D$ is Cartier and base point free. 
\end{theorem} 

\begin{proof} 
{\it Step 1. $\vert D \vert $ has no fixed component. }

Arguing by contradiction we assume that the linear system $|D|$ has a fixed component, i.e. we have a decomposition
$$
D \simeq M + B
$$
into mobile and fixed part with $B \neq 0$. Following Notation \ref{notation-Qfactorial},
we consider a $\Q$-factorialisation $\holom{\eta}{X'}{X}$ and denote by
$$
D' \simeq M' + B'
$$
the induced decomposition into fixed and mobile part of $D' = \mu^* D$.
We apply Lemma \ref{lemma-indeterminacies} to the Weil $\Q$-Cartier divisor
$M$, so we get a birational morphism $\holom{\mu}{Y}{X'}$
such that
$$
\mu^* M' \sim_\Q M_Y + \sum_{i=1}^k a_i E_i
$$
where $a_i>0$ and the $E_i$ are $\mu$-exceptional.

By assumption the map $\varphi_{|D|}$ is generically finite onto its image, 
hence the  map $\varphi_{|D'|}=\varphi_{|M'|}$ is generically finite onto its image.
Since $\varphi_{|M_Y|}$ is a resolution of indeterminacies of $\varphi_{|M'|}$,
the divisor $M_Y$ is nef and big.
Therefore by \cite[Thm.1.1]{Fuj89} we have  
$$
\Delta(Y, M_Y) =  n+ M_Y^n - h^0(Y,\sO_Y(M_Y)) \geq 0.
$$ 
Hence
\begin{equation} \label{eq:eq} 
M_Y^n \geq  h^0(Y,\sO_Y(M_Y)) - n = h^0(X,\sO_X(D)) - n \geq D^n. 
\end{equation}

Set $D'^n=D^n =: a.$ 
We will now use the Hodge index theorem applied to the nef divisors 
$\mu^* D'$ and $M_Y$:  by \cite[Cor.2.5.3]{BS95} we have for all $0 \leq k \leq n$ that
$$
\left( (\mu^* D')^k M_Y^{n-k} \right)^n \geq \left((\mu^* D')^n \right)^k \cdot (M_Y^n)^{n-k}
= a^k  (M_Y^n)^{n-k}.
$$
In particula by \eqref{eq:eq} we obtain 
$$
((\mu^* D')^{n-1} M_Y)^n \geq  a^{n-1}  (M_Y^n) \geq a^n.
$$
Now observe that we can represent some positive multiple of $(\mu^* D')^{n-1}$ by a curve $C$ that is disjoint from the image of the exceptional locus of $\mu$. In a neighbourhood of this curve the divisor $M_Y$ coincides
with $\mu^* M'$. Thus we have
$$
((\mu^* D')^{n-1} \cdot M_Y)^n
=
(\mu^* D')^{n-1} \cdot (\mu^* M')
< (\mu^* D')^{n} = a
$$
where the strict inequality follows from  $D'^{n-1} \cdot B>0$ (cp. Remark \ref{remark-rather-big}) and the projection formula. Thus we have reached a contradiction.

{\it Step 2. The conclusion.} 
We start with similar considerations as in Step 1. 
By Lemma \ref{lemma-indeterminacies} we can find a modification $\holom{\mu}{Y}{X}$
such that $\mu^* D \sim_\Q M_Y + E$ where $M_Y$ is a globally generated Cartier divisor
that defines a resolution of indeterminacies of $\varphi_{|D|}$, and $E$ is an effective $\mu$-exceptional $\Q$-Cartier divisor. Since we assume that $\varphi_{|D|}$ is generically finite onto its image,
the divisor $M_Y$ is big. As in Step 1 we know by Fujita's result \cite[Thm.1.1]{Fuj89} that
$M_Y^n \geq D^n$. Since both $M_Y$ and $D$ are nef, we have
$$
\mu^* D^{n-k} \cdot M_Y^{k-1} \cdot E \geq 0 
$$
and thus  
$$
\mu^* D^{n-k} \cdot M_Y^{k} \geq \mu^* D^{n-k-1} \cdot M_Y^{k+1}.
$$
Thus we also have $D^n \geq M_Y^n$ and all the preceding inequalities are equalities.

We will now show that $E=0$. Arguing by contradiction we assume $E \neq 0$, and 
denote by $l$ the largest dimension of an irreducible component of $\mu(\supp E)$. 
Fix an irreducible component $E_i \subset \supp E$ such that $\mu(E_i)$ has dimension $l$.
Since 
$\mu^* D^{l} \cdot M_Y^{n-l-1} \cdot E=0$ we have $\mu^* D^{l} \cdot M_Y^{n-l-1} \cdot E_i=0$.
Moreover $(\mu^* D)^{l} \cdot E_i$ is represented (up to positive multiples) by a general fibre of $E_i \rightarrow \mu(E_i)$. Since $M_Y$ is $\mu$-ample by Lemma \ref{lemma-indeterminacies}
we get $(\mu^* D)^{l} \cdot E_i \cdot M_Y^{n-l-1} > 0$, a contradiction.

This shows that $E = 0$, so again by Lemma \ref{lemma-indeterminacies} $\mu$ is an isomorphism. Hence $D \simeq M_Y$ is Cartier and base point free by \cite[Thm.4.2]{Fuj90}. 
\end{proof}

\begin{remark}
In the situation of Theorem \ref{thm:delta-0} much more can be said: since $D$ is ample and {\em Cartier}, we can apply \cite[Thm.4.2]{Fuj90} to see that the polarised variety $(X,D)$ has $\Delta$-genus zero, i.e. we have
$$
h^0(X,\sO_X(D)) = n + D^n.
$$
These varieties are classified by \cite[Thm.5.10,Thm.5.15]{Fuj90}, in particular we have: 
\end{remark} 

\begin{corollary}  \label{cor:quadric}  
In the situation of Theorem \ref{thm:delta-0}, the following holds:
\begin{itemize}
\item If $D^n=1$, then $X \simeq \PP^n$.
\item If $D^n = 2$, then $X$ is isomorphic to a quadric.
\end{itemize} 
\end{corollary}

\section{Deformations of $\mathbb P^n$ - the general setup} 
\label{section-setup}

\begin{definition} \label{def1} Let $X$ be a normal projective variety with klt (resp. canonical, resp. terminal) singularities. Then $X$ is said to be a klt  degeneration of $\mathbb P^n$,
if there exists a normal complex space $\mathcal X$ and a projective morphism 
$$ \pi: \mathcal X \to \Delta $$
to a disc $\Delta \subset \mathbb C$, such that 
$$ 
X_t = \pi^{-1}(t) \simeq \mathbb P^n
$$
for $t \ne 0$ and $X_0 \simeq X$.
\end{definition} 

\begin{remark} \label{remark-def1}
Since $\rho(\PP^n)=1$ the total space $\mathcal X$ is $\Q$-Gorenstein \cite[Thm.3.4]{HN13}.
Moreover the central fibre $X$ being klt, the total space $\mathcal X$ is also klt \cite[Thm.7.5]{Kol97}. Note that we do not assume that $X$ is $\Q$-factorial, since this condition is not easy to check in practice. 
\end{remark}

The basic examples of degenerations of $\PP^n$ are obtained by a cone construction which
we learned from \cite[Example 1]{AR14}: 

\begin{example} \label{ex:basic} 
Let $\holom{v_d}{\PP^n}{\PP^N}$ the $d$-th Veronese embedding of $\PP^n$,
and let $W \subset \PP^{N+1}$ be the cone over $v_d(\PP^n)$. Let
$\mathcal X \rightarrow W$ be the blowup of the base locus of a general pencil of hyperplane sections, then we have a flat fibration 
$$
\pi: \mathcal X \rightarrow \PP^1
$$
such that the general fibre is isomorphic to $\PP^n$, and the special fibre $X=\mathcal X_0$ is a cone over a hyperplane section $Z$ of $v_d(\PP^n)$. Note that this means that
$Z \in |\sO_{\PP^n}(d)|$.
As an abstract variety we obtain $X$ by considering the projectivised bundle
$$
q: \PP(\sO_Z \oplus \sO_Z(-d)) =:Y \to Z
$$
and contracting the negative section $E := \PP(\sO_Z(-d))$ onto a point.
Note that
$$
-K_Y  \simeq q^* c_1(\sO_Z(n+1)) + 2 E,
$$
and 
$$
- K_X =  \mu_* (-K_{Y}) = (n+1) \mu_* q^* c_1(\sO_Z(1)),
$$
hence $K_X$ has Weil index $n+1$. Since $\sO_{E}(E) \simeq \sO_Z(-d)$
we compute easily that
$$
K_{Y}  \sim_\Q \mu^* K_X + \frac{n+1-2d}{d} E,
$$
so $X$ has klt (resp. canonical, resp. terminal) singularities if and only if 
$d<n+1$ (resp. $d \leq \frac{n+1}{2}$, resp. $d < \frac{n+1}{2}$).
A direct computation shows that  $\mu_*(T_{Y/Z}) \subset T_X$ destabilizes $T_X$ for $d \geq 2$, hence $T_X $ is never semistable. Theorem \ref{theorem-stability} shows that this is no coincidence.
\end{example}

\begin{proposition} \label{proposition-picard-number} Let $X$ be a klt degeneration of $\mathbb P^n$ (or more generally of a Fano manifold with Picard number one). Then $X$ is $\mathbb Q$-Fano with $\rho(X) = 1$.
\end{proposition}

\begin{remark*}
The proof of the proposition also shows that the anticanonical divisor $-K_{\mathcal X}$ is relatively ample and $\rho(\mathcal X/\Delta)=1$.
\end{remark*}

\begin{proof} 
Since $\pi$ is a projective morphism, there exists a relatively ample Cartier divisor $A \rightarrow \mathcal X$. 
Since $\rho(X_t)=1$ for $t \ne 0$, we can find positive integers $p,q \in \N$ such that for a general fibre
$$
\sO_{X_t}(p A + q K_{X_t}) \simeq \sO_{X_t}
$$
Since $K_{\mathcal X}$ is $\Q$-Cartier we can assume, up to taking multiples,
that $\sO_{\mathcal X}(p A + q K_{\mathcal X})$ is locally free. 
Since $\pi$ is flat and $\sO_{\mathcal X}(p A + q K_{\mathcal X})$ is locally free we know by 
semicontinuity \cite[Thm.III.4.12]{BS76}
that
$h^0(X_0, \sO_{X_0}(p A + q K_{X_0})) \neq 0$. Yet the intersection
number $X_0 \cdot c_1(pA+q K_{\mathcal X}) \cdot c_1(A)^{\dim X-1}$
is zero, so the zero locus of any section $s  \in H^0(X_0, \sO_{X_0}(p A + q K_{X_0}))$ is empty. This finally shows that  $\sO_{X_0}(p A + q K_{X_0}) \simeq \sO_{X_0}$ and $X_0$ is Fano.

After shrinking $\Delta$, we may assume that $X_0$ is a deformation retract of $\mathcal X$
\cite[I,Thm.8.8]{BHPV04}.  Hence the restriction map
$$ H^q(\mathcal X, \mathbb Z) \to  H^q(X_0, \mathbb Z) $$
is an isomorphism for all $q \in \N$. Since $\mathcal X$ is klt (cf. Remark \ref{remark-def1})
and $-K_{\mathcal X}$ is relatively ample we know by relative Kodaira vanishing that
$R^j \pi_* \sO_{\mathcal X}=0$ for all $j \geq 1$. Thus the natural map
$$
\pic(\mathcal X) \rightarrow H^2(\mathcal X, \mathbb Z)
$$
is surjective.
Assuming $\rho(X_0) \ge 2$, there exists a non-isomorphic extremal contraction
$$ 
\varphi_0: X_0 \to Y_0
$$
to a normal projective variety with $\dim Y_0 > 0$.  Choose an ample Cartier divisor $H_0$ on $Y_0$ and set $\sL_0 = \varphi^*(H_0)$.
by what precedes $\sL_0$ extends  to a holomorphic line bundle $\sL$ on $\mathcal X$ and let $\sL_t$ be its restriction to $X_t$.

Again since $\rho(X_t)=1$ we can, up to replacing $\sL$ by a multiple,
assume that
$$
\sL_t \simeq \omega_{X_t}^{\otimes -r}
$$
for $t \ne 0$. As in the first step we know by semicontinuity that
$$
h^0(X_0, \sL_0 \otimes \omega_{X_0}^{-r}) \neq 0.
$$
Since the intersection number $X_0 \cdot c_1(\sL \otimes \omega_{\mathcal X}^{-r}) \cdot c_1(A)^{\dim X-1}$ is zero we obtain $\sL_0 \simeq \omega_{X_0}^{-r}$.
Since $X_0$ is Fano this implies that $\sL_0$ is ample, a contradiction to the construction
of $\sL_0$. 
\end{proof}

The following lemma is well-known to experts, for the convenience of the
reader we reproduce the proof from \cite[Lemma 2.1]{EV85}:

\begin{lemma} \label{lemma-semicontinuity}
Let $\mathcal X$ be a normal complex space and let $\holom{\pi}{\mathcal X}{\Delta}$ be a projective fibration onto a smooth complex curve $\Delta$. Let $\sF$ be a reflexive sheaf on $\mathcal X$.  If $\pi$ has reduced fibres, the function
$$
t \rightarrow h^0(\mathcal X_t, (\sF \otimes \sO_{\mathcal X_t})^{**})
$$
is semicontinuous.
\end{lemma}

\begin{proof}
Fix a relative ample line bundle $\sO_\mathcal X(1)$. Then $\sF(m) = \sF \otimes \sO_\mathcal X(m)$ is reflexive for any  $m \in \N$,  hence torsion-free. Therefore its direct image $f_* (\sF(m))$ is torsion-free as well, hence locally free.
Since $\Delta$ is a smooth curve,
$\sF$ is flat over $\Delta$, as a direct consequence of Hironaka's flattening theorem \cite{Hir75} (see also \cite[Thm.3.20]{Kol22}). 
Therefore the function
$$
t \rightarrow h^0(\mathcal X_t, \sF \otimes \sO_{\mathcal X_t})
$$
is upper semi-continuous by \cite[Thm.III.4.12]{BS76}. 
For a general fibre $\mathcal X_t$ the restriction $\sF \otimes \sO_{\mathcal X_t}$ is reflexive \cite[Thm.12.2.1]{Gro66}, so we have
$$
h^0(\mathcal X_t, \sF \otimes \sO_{\mathcal X_t}) = h^0(\mathcal X_t, (\sF \otimes \sO_{\mathcal X_t})^{**})
$$
except for finitely many $t \in \Delta$. Thus we are done if we show that for {\em every} fibre $\mathcal X_0$
the sheaf  $\sF \otimes \sO_{\mathcal X_0}$ is torsion-free. Indeed, $X_0$ being reduced,
we will then get an inclusion  $\sF \otimes \sO_{\mathcal X_0} \hookrightarrow (\sF \otimes \sO_{\mathcal X_0})^{**}$ and 
an inequality
$$
h^0(\mathcal X_0, \sF \otimes \sO_{\mathcal X_0}) \leq h^0(\mathcal X_0, (\sF \otimes \sO_{\mathcal X_0})^{**}).
$$

{\em Proof of the claim.}
Let 
$$
i: W \hookrightarrow \mathcal X
$$
be the locus where $\sF$ is locally free, then 
$\mathcal X \setminus W$ has codimension at least two. 
Set $W_0:=W \cap \mathcal X_0$, and let 
$$
\holom{i_0}{W_0}{\mathcal X_0}
$$
be the inclusion. Since $\mathcal X \setminus W$ has codimension at least two and
$\mathcal X_0$ has pure codimension one, the morphism $i_0$ is dominant.

Since $\sO_{\mathcal X}(-\mathcal X_0)$ is locally free and $i_* i^* \sF \simeq \sF$ by \cite[Prop.1.6]{Har80} we obtain from the projection formula that
$$
i_* (i^*(\sF \otimes \sO_W(-W_0))) \simeq \sF \otimes \sO_{\mathcal X}(-\mathcal X_0).
$$
Since $\sF$ is locally free on $W$ we have an exact restriction sequence
$$
0 \rightarrow i^*(\sF \otimes \sO_W(-W_0)) \rightarrow i^* \sF \rightarrow i_0^* \sF \rightarrow 0.
$$
Pushing forward by $i$ we obtain an exact sequence
$$
0 \rightarrow \sF \otimes \sO_{\mathcal X}(-\mathcal X_0) \rightarrow \sF \rightarrow (i_0)_* i_0^* \sF.
$$
On the other hand, by right exactness of the tensor product and torsion-freeness of $\sF \otimes \sO_{\mathcal X}(-\mathcal X_0)$ we have an exact sequence
$$
0 \rightarrow \sF \otimes \sO_{\mathcal X}(-\mathcal X_0) \rightarrow \sF \rightarrow \sF \otimes \sO_{\mathcal X_0} \rightarrow 0.
$$
Thus we have an injection $\sF \otimes \sO_{\mathcal X_0} \rightarrow (i_0)_* i_0^* \sF$.
Yet the latter is torsion-free since $i_0$ is dominant.
\end{proof}

\begin{proposition} \label{prop:semi}  Let $X$ be a klt degeneration of $\PP^n$. Then there exists a reflexive sheaf $\sO_{\mathcal X}(1) $ on $\mathcal X$  with the following properties.
\begin{itemize}
\item $\sO_{\mathcal X}(1) \otimes \sO_{X_t} = \sO_{X_t}(1) $ for all $t \ne 0$;
\item $\sO_{X_0}(1) := (\sO_{\mathcal X}(1)  \otimes \sO_{X_0})^{**} $ is a reflexive sheaf with $$
h^0(X_0,\sO_{X_0}(1)) \geq n+1;
$$
\item $\sO_{\mathcal X}(-n-1) := (\sO_{\mathcal X}(1)^{\otimes (-n-1)})^{**} \simeq \omega_{\mathcal X}$.
\end{itemize}
In particular if $\mathcal H$ is the Weil $\Q$-Cartier class associated to $\sO_{\mathcal X}(1)$,
it is relatively ample.
\end{proposition} 

\begin{proof}  
By Proposition \ref{proposition-picard-number}, the divisor $-K_{\mathcal X}$ is $\pi$-ample.

Set $\mathcal X^* = \mathcal X \setminus X_0$. Since $X_0$ is reduced, we can choose (up to restricting $\Delta$) local sections $s_1, \ldots, s_{n-1}$ spanning a hyperplane
in a general fibre. Thus there is a line bundle $\sL$ on $\mathcal X^*$  such that $\sL \otimes \sO_{X_t} = \sO_{X_t}(1)$ for all $t \ne 0$. 
Choose $D \in \vert \sL \vert$ and set $D_t = D \cap X_t$ for $t \ne 0$ 
Then $$ D_t \cdot (-K_{X_t})^{n-1} = (n+1)^{n-1}  $$ is constant, hence  by \cite{Bis64}  the closure $\overline D $ of $D$ is analytic set in $\mathcal X$.
 Now let $D_0 = \overline D \cap X_0$ and
set 
$$ \sO_{\mathcal X}(1)  = \sO_{\mathcal X}(\overline D).$$ 
Then $(\sO_{\mathcal X}(1) \otimes \sO_{X_0})^{**} = \sO_{X_0}(D_0)$ since the two sheaves
coincide on $X_{0,nons}$ and $X_0$ is normal. Thus we obtain 
$h^0(X_0, \sO_{X_0}(1)) \geq n+1$ by Lemma \ref{lemma-semicontinuity}. 
The last claim is clear.
\end{proof} 

\begin{proposition} \label{prop:easy} 
Let $X$ be a klt degeneration of $\PP^n$. If $\sO_X(1)=\sO_X(H)$ is locally free, then 
$X \simeq \PP^n$. 
\end{proposition} 

\begin{proof}
This can be shown using Wahl's theorem \cite{Wah83} and semicontinuity for the sheaf
of Euler vector field $T_{\mathcal X/\Delta} \otimes \sO_{\mathcal X}(-1)$. 

More directly note that by Proposition \ref{prop:semi} the $\Delta$-genus of $(X,H)$
is not positive. Thus by \cite[Thm.4.2]{Fuj90} the $\Delta$-genus is zero and $|H|$
is basepoint-free, defining an isomorphism onto $\PP^n$.
\end{proof}

\section{The rational fibration} 
\label{section-rational}

We fix the setup for this section:

\begin{setup} \label{setup-degen} {\rm Let
$$
\holom{\pi}{\mathcal X}{\Delta}
$$
be a klt degeneration of the projective space $\mathbb P^n$.
By Proposition \ref{prop:semi}  the central fibre
$X$ comes with a Weil $\Q$-Cartier divisor $H$ such that $H^n=1$ and 
$$
h^0(X, \sO_X(H)) \geq n+1.
$$ 
In what follows we will use the notation $\sO_X(1) = \sO_X(H)$ for the reflexive sheaf
correspond to $H$.
}
\end{setup}
We start with a consequence of Theorem \ref{thm:delta-0}:

\begin{corollary} \label{cor:gen-finite} Let $X$ be a klt degeneration of $\mathbb P^n$ and 
$$
\merom{f}{X}{\PP^N}
$$
be the rational morphism defined by the linear system $|\sO_{X}(1)|$.
If $f$ is generically finite onto its image, then $X \simeq \PP^n$. 
\end{corollary} 

\begin{proof}
By Proposition \ref{prop:semi} we have $h^0(X, \sO_X(1)) \geq n+1$. Since $f$ is generically finite onto its image, we know by Theorem \ref{thm:delta-0} that $H$ is Cartier.
Thus we conclude by Corollary \ref{cor:quadric}. 
\end{proof}

The semicontinuity lemma \ref{lemma-semicontinuity} applies to the reflexive sheaf
$(T_{\mathcal X/\Delta} \otimes \sO_{\mathcal X}(-1))^{**}$, so the Euler vector fields
on $\PP^n$ degenerate to vector fields on the central fibre vanishing along an element of $|H|$. Let us first consider an instructive example:

\begin{example} \cite[Ch.8]{Pin74} \cite[Ex.3.3]{HN13} \label{example-degeneration-cone}
Consider Example \ref{ex:basic} in the case $n=2, d=2$, i.e. a family of projective planes
degenerating to the cone $X$ obtained by contracting the negative section $E$ of $\mathbb F_4$. If $\holom{\mu}{\mathbb F_4}{X}$ is the contraction, we have $\sO_X(1) = \mu_* \sO_{\mathbb F_4}(2F)$ and $\sO_X(1)$ is $2$-Cartier but not Cartier. The minimal resolution $\mu$ is functorial \cite[Prop.1.2]{BW74}, so every vector field lifts to $\mathbb F_4$. Thus we have
$$
H^0(X, T_X(-1)) \hookrightarrow H^0(\mathbb F_4, T_{\mathbb F_4} \otimes \sO(-2F))
$$
Let $\holom{\pi}{\mathbb F_4}{\PP^1}$ be the ruling, and consider the twisted tangent sequence
$$
0 \rightarrow T_{\mathbb F_4/\PP^1}(-2F) \rightarrow T_{\mathbb F_4}(-2F) \rightarrow
\pi^* T_{\PP^1}(-2F) \simeq \sO_{\mathbb F_4} \rightarrow 0. 
$$ 
Since $T_{\mathbb F_4}$ is not a direct sum we have 
$$
H^0(\mathbb F_4, T_{\mathbb F_4/\PP^1}(-2F)) = H^0(\mathbb F_4, T_{\mathbb F_4}(-2F)),
$$
so every Euler vector field on $X$ is tangent to the ruling of the cone, so the closures of general leaves are lines $l_y$. Note that this implies that $\sO_X(1) \cdot l_y = \frac{1}{2}$.

Fix now a general point $x \in X$ and an Euler vector field vanishing in $x$. For a general point $y \in X$, the line $l_y$ does not pass through $x$. However if we take the union of the lines $l_x+l_y$, they connect $x$ and $y$, and arise as a degeneration of a line in $\PP^2$ (i.e. as the degeneration of the closure of a leaf of an Euler vector field on $\PP^2$).
\end{example}

In what follows we will not work with the degenerate Euler fields on $X$, but rather with 
the more algebraic counterpart, the degenerations of lines: in the situation of Setup \ref{setup-degen}, choose a general point $x \in X$ and a $\pi$-section $s: \Delta \rightarrow \mathcal X$
passing through $x$. For $t \neq 0$, the family of lines $l_t \subset X_t$
passing through $s(t)$ is a covering family of $X_t$, so taking the closure
in the relative cycle space $\mathcal C(\mathcal X/\Delta))$ determines a covering family of effective cycles
$\sum b_j l_j$
on $X$ such that $x \in \supp \sum b_j l_j$ and
\begin{equation} \label{eqn-same-class}
l_t = \sum b_j l_j  \in H_2(\mathcal X, \Z).
\end{equation} 
Choose now $y \in X$ a general point, then there exists a
cycle $\sum b_j l_j$ passing through $x$ and $y$, 
and we denote by $l \subset \sum b_j l_j$ an irreducible component passing through $y$. Since $y \in X$ is general, the curves $l$ form a dominating family of rational curves of $X$
and we call $l$ a {\it degenerate line.}

\begin{theorem} \label{theorem-euler-fields}
Let $X$ be a klt degeneration of $\PP^n$ and
$$
\merom{f}{X}{\PP^N}
$$
be the rational morphism defined by the linear system $|\sO_{X}(1)|$.
If $X \not\simeq \PP^n$, then $f$ is not generically finite onto its image and the degenerate lines $l$ are contained the fibres of $f$.
\end{theorem}

\begin{remark*}
Analogously to Example \ref{example-degeneration-cone} we can show that the degenerate lines
on $X$ are the closures of leaves of Euler vector fields on $X$. Thus an equivalent formulation
of the statement is that that the Euler vector fields are tangent to the $f$-fibres.
\end{remark*}

\begin{proof}
Since $X \not\simeq \PP^n$, we know by Corollary \ref{cor:gen-finite} that $f$ is not generically finite onto its image.
Since $H$ is ample and $l \subset \sum b_j l_j$ it follows by \eqref{eqn-same-class} that
$$
H \cdot l \leq 1,
$$
and equality holds if and only if $l=\sum b_j l_j$, i.e., $l$ is a rational curve passing through $x$ and $y$.

Following Notation \ref{notation-Qfactorial}, 
let $\holom{\eta}{X'}{X}$ be a $\Q$-factorialisation of $X$ and
$$
H' \simeq M' + B'
$$
be the decomposition into mobile and fixed part of $H' = \eta^* H$.
Let $l' \subset X'$ be the strict transform of $l$. Since the curves $l'$ dominate $X'$ we have
$$
M' \cdot l' \leq H' \cdot l' = H \cdot l \leq 1.
$$
Let now $\holom{\mu}{Y}{X'}$ be the modification from Lemma \ref{lemma-indeterminacies} and
$$
\mu^* M' \sim_\Q M_Y + \sum_{i=1}^k a_i E_i
$$
be the decomposition given by the lemma. The morphism $\varphi_{|M_Y|}$ associated with the linear system $\vert M_Y \vert$ 
coincides with $f$ on a Zariski open set. 
Let $l_Y \subset Y$ be the strict transform of $l'$, then
$$
1 \geq M' \cdot l' =  \mu^* M' \cdot l_Y = (M_Y + \sum_{i=1}^k a_i E_i) \cdot l_Y.
$$
Since $l_Y$ is not contained in any exceptional divisor, 
 $$(\sum_{i=1}^k a_i E_i) \cdot l_Y \geq 0,$$ hence $M_Y \cdot l_Y \leq 1$. Yet $M_Y$ is a nef {\em Cartier} divisor, so there are  only
 two cases:  either $M_Y \cdot l_Y=0$ and $l_Y$ is contracted by $\varphi_{|M_Y|}$,
or $M_Y \cdot l_Y=1$. 
In the former case $l$ is contracted by $f$, which is our claim. 
Thus it remains to derive a contradiction assuming that $M_Y \cdot l_Y=1$.
Then all the inequalities above are equalities, in particular
we have $H \cdot l=1$.

As mentioned at the beginning of the proof, this implies that the degenerate line $l$ passes through both $x$ and $y$.
Consider the surjective map
$$
\holom{\varphi_{|M_Y|}}{Y}{Z}.
$$
Since $x \in X$ is general, $x$ can be considered as a point on $Y$ and its image $\varphi_{|M_Y|}(x)$
is in the smooth locus of $Z$. Further,  the irreducible curves $\varphi_{|M_Y|}(l_{Y})$ passing through $\varphi_{|M_Y|}(x)$ dominate $Z$. Writing $M_Y = \varphi_{|M_Y|}^*\sO_Z(1)$, it follows 
$$
1 = M \cdot l_{Y} = c_1(\sO_Z(1)) \cdot (\varphi_{|M_Y|})_*(l_{Y}).
$$
Since $\sO_Z(1)$ is ample, 
$$c_1(\sO_Z(1))^{\dim Z} \leq 1$$
 by \cite[V,Prop.2.9]{Kol96}.
 On the other hand 
$$h^0(Z, \sO_Z(1))=h^0(Y, \sO_Y(M)) = h^0(X, \sO_X(D)) \geq n+1.$$ Since $\sO_Z(1)$ is Cartier, the non-negativity of the sectional genus \cite[Thm.4.2]{Fuj90}
$$
\Delta(Z, \sO_Z(1)) := \dim Z + c_1(\sO_Z(1))^{\dim Z} - h^0(Z, \sO_Z(1))
$$
yields $\dim Z=n$. Since $f$ is not generically finite, this is a contradiction.
\end{proof}

\begin{remark} \label{remark-degree-one}
Following the strategy of proof \cite[4.4]{BBI00},
\cite[Thm.4.2]{IN03}, we can show a variant of Theorem \ref{theorem-euler-fields}:
let $X$ be a klt degeneration of $\PP^n$, and let $l$ be a degenerate line on $X$.
If $H \cdot l=1$, then $X \simeq \PP^n$.
\end{remark}

The next result shows  characterises Example \ref{ex:basic}  via the rational fibration:

\begin{theorem} \label{theorem-case-nminusone}
Let $X$ be a canonical degeneration of $\PP^n$ and
$$
\merom{f}{X}{\PP^N}
$$
be the rational morphism defined by the linear system $|\sO_{X}(1)|$.
Denote by $Z \subset \PP^N$ the closure of the image of $f$, and let $d$ be its degree. 

If $\dim Z = n-1$, then $2 \leq d \leq \frac{n+1}{2}$.
\end{theorem}

\begin{proof}
Following Notation \ref{notation-Qfactorial}, 
let $\holom{\eta}{X'}{X}$ be a $\Q$-factorialisation of $X$ and
$$
H' \simeq M' + B'
$$
be the decomposition into mobile and fixed part of $H' = \eta^* H$.

Applying Lemma \ref{lemma-indeterminacies} to $M'$ we obtain a birational morphism
$\holom{\mu}{Y}{X}$ and a decomposition 
$$
\mu^* M' \sim_\Q M_Y + \sum_{i=1}^k a_i E_i
$$
such that $a_i > 0$ and 
$$
\holom{\varphi:=\varphi_{|M_Y|}}{Y}{Z}
$$
resolves the indeterminacies of $f$.
Note that $d \geq 2$ since otherwise $\varphi(Y) \subsetneq \PP^N$ is linearly degenerate
and therefore the map $H^0(\PP^N, \sO_{\PP^N}(1)) \rightarrow H^0(X, \sO_X(1))$ not injective.

Our goal is to show that 
$$
d= c_1(\sO_Z(1))^{n-1} \leq \frac{n+1}{2}.
$$
Observe first that
$$
\mu^* H' \sim_\Q M_Y + \sum_{i=1}^k a_i E_i + \mu^* B',
$$
so we have $\mu^* H' \sim_\Q M_Y + A$ with $A:=\sum_{i=1}^k a_i E_i + \mu^* B'$ an effective $\Q$-divisor. Since $\mu^* H'$ and $M_Y$ are nef, 
we argue inductively to obtain inequalities of cycles
\begin{equation}
\label{eqn1}
(\mu^* H')^{n-1} \geq (\mu^* H')^{n-2} \cdot M_Y \geq \ldots \geq M_Y^{n-1}.
\end{equation}

Now recall that $M_Y = \varphi^*\sO_Z(1)$ and that $f$ contracts the degenerate lines $l \subset X$
by Theorem \ref{theorem-euler-fields}.
Thus if $l_Y \subset Y$ is the strict transform
and $e$ the number of connected components of a general $\varphi$-fibre, we have 
$$
M_Y^{n-1} = d \cdot e \ l_Y.
$$
Hence by \eqref{eqn1}
\begin{equation} \label{upperbound}
H \cdot (\eta \circ \mu)_* l_Y =
(\mu^* H') \cdot l_Y
=
\frac{1}{d \cdot e} (\mu^* H') \cdot M_Y^{n-1} \leq \frac{1}{d \cdot e} (\mu^* H')^n = \frac{1}{d \cdot e} \leq \frac{1}{d}. 
\end{equation}
On the other hand, since $X$ has canonical singularities we have
$$
K_Y \sim_\Q \mu^* \eta^* K_X + \sum_{i=1}^k b_i E_i
$$
such that $b_i \geq 0$. 
Since $Y$ is normal, $Y$ is smooth in a neighbourhood of the general curve $l_Y$ which is a connected component of a fiber of $\varphi$.  In particular, the intersection numbers $K_Y \cdot l_Y$ and $E_i \cdot l_Y$ are well-defined and
$$
K_Y \cdot l_Y=-2, \qquad E_i \cdot l_Y \geq 0 \qquad \forall \ i=1, \ldots, k
$$
Therefore
\begin{equation}
\label{eqn2}
-2 = K_Y \cdot l_Y + \sum_{i=1}^k b_i E_i \cdot l_Y \geq \mu^* \eta^* K_X \cdot l_Y, 
\end{equation}
Combining with \eqref{upperbound} we thus obtain
\begin{equation}
\label{eqn3}
2 \leq -K_X \cdot (\eta \circ \mu)_* l_Y = (n+1) H \cdot (\eta \circ \mu)_* l_Y \leq \frac{n+1}{d},
\end{equation}
which proves our assertion. 
\end{proof}

\begin{corollary} \label{corollary-nminusone}
In the situation of Theorem \ref{theorem-case-nminusone}, assume that
$$
d=\frac{n+1}{2}.
$$
Then the  linear system $|\sO_X(1)|$ is without fixed components.
Moreover the following holds: 
\begin{enumerate}
\item the modification $\holom{\mu}{Y}{X}$ is crepant;
\item for every $\mu$-exceptional divisor $E_i$, the image $\mu(E_i)$ is a point
and $\varphi|_{E_i} : E_i \rightarrow Z$ is finite;
\item the morphism $\varphi$ has connected fibres and is equidimensional;
\item the anti-canonicaldivisor $-K_Y$ is $\varphi$-ample.
\end{enumerate}
\end{corollary}

\begin{proof} In case $d=\frac{n+1}{2}$, we have equality 
in all the inequalities in the proof
of Theorem \ref{theorem-case-nminusone}; and we spell this out: 

If the fixed part $B$ of $|\sO_X(1)|$ is not empty, we have $H' = M' + B'$ with $B' \neq 0$
and therefore 
$$
\mu^* (H')^{n-1} \cdot A \geq (H')^{n-1} \cdot B' >0
$$
by Remark \ref{remark-rather-big}.
Thus
$(\mu^* H')^n > (\mu^* H')^{n-1} M_Y$ and therefore 
$(\mu^* H')^n>(\mu^* H') \cdot M_Y^{n-1}$ by \eqref{eqn1}. 
Yet this contradicts equality in \eqref{upperbound}, so we get $B=0$.

Note that now the proof of  Theorem \ref{theorem-case-nminusone} holds without changes for $X' = X$ and $M' = H$. 

Since $B=0$ we can apply Lemma \ref{lemma-indeterminacies} to the mobile divisor $H$
on $X$, the proof of Theorem \ref{theorem-case-nminusone} hold without taking a $\Q$-factorialisation $\eta$. In particular the $\mu$-exceptional locus coincides with
$\supp(\sum_{i=1}^k a_i E_i)$.

 As already mentioned, all the inequalities in the previous proof are equalities, in particular 
\eqref{eqn3} becomes $$H \cdot \mu_* l_Y = \frac{2}{n+1},$$
and from \eqref{eqn2} we derive that for every exceptional divisor $E_i$ we have $E_i \cdot l_Y=0$
or $b_i=0$. From the equality in \eqref{upperbound} it follows $e=1$ (so $\varphi$ has connected fibres) and
$$
M_Y^{n-1} = \frac{n+1}{2} l_Y.
$$
Hence 
$$
\mu^* H \cdot M_Y^{n-1} = 1.
$$
We claim that $\mu(E_i)$ is a point for every $\mu$-exceptional divisor $E_i$.

{\em Proof of the claim.}
Arguing by contradiction, let $p>0$ be the maximal dimension of an irreducible component of $\mu(\supp(\sum_{i=1}^k a_i E_i))$. Up to renumbering we may assume $\dim \mu(E_1)=p$. Since 
$H$ and $M_Y$ are nef, we have
$$
(\mu^* H)^{n} \geq \mu^* H \cdot M_Y^{n-1} + M_Y^{n-1-p} \cdot (\mu^* H)^{p} \cdot (\sum_{i=1}^k a_i E_i).
$$
Since $H$ is ample, the 
class of $(\mu^* H)^{p} \cdot E_i$ is a positive multiple of a general fibre of
$E_1 \rightarrow \mu(E_1)$. Since $M_Y$ is $\mu$-ample by Lemma \ref{lemma-indeterminacies},
we obtain 
$$
M_Y^{n-1-p} \cdot (\mu^* H)^{p} \cdot (\sum_{i=1}^k a_i E_i) \geq
a_1 M_Y^{n-1-p} \cdot (\mu^* H)^{p} \cdot E_1 
>0,
$$
a contradiction to $H^n=1=(\mu^* H) \cdot M_Y^{n-1}$. This shows the claim.

Now that we know that $\mu(E_i)$ is a point, by the $\mu-$ampleness of $M_Y$, the restriction
$M_Y|_{E_i} \simeq \varphi^* \sO_Z(1)|_{E_i}$ is ample. Thus $\varphi|_{E_i}$ is finite,
 concluding the proof of b). Since $E_i \twoheadrightarrow Z$ we have $E_i \cdot l_Y>0$
and therefore $b_i=0$ for every exceptional divisor. Thus we have shown a).

Let us next show c), i.e.,  $\varphi$ is equidimensional. If $F$ is any $\varphi$-fibre,
the intersection 
$$
F \cap \mbox{Exc}(\mu) = F \cap \supp(\sum_{i=1}^k a_i E_i)
$$
is finite. Thus $F \rightarrow \mu(F)$ is finite, so $\mu^* H$ is $\varphi$-ample. Since
$$
\mu^* H|_F \sim_\Q (M_Y + \sum_{i=1}^k a_i E_i)|_F = \sum_{i=1}^k a_i ({E_i}_{\vert F})
$$
is ample, and since the sets $E_i \cap F$ are finite by $b)$, the fibre $F$ has pure dimension one.

Finally note that by a) we have 
$$
-K_Y \sim_\Q \mu^*(-K_X) \sim_\Q (n+1) \mu^* H.
$$
We have just shown that $\mu^* H$ is $\varphi$-ample, so $d)$ follows.
\end{proof}

We will use the following variant of Corollary \ref{corollary-nminusone},
the proof works without any changes:

\begin{lemma} \label{lemma-variant}
Let $X$ be a canonical degeneration of $\PP^n$.
For some $k \in \N$, let $\merom{g}{X}{\PP^N}$
be the rational morphism defined by the linear system $|\sO_{X}(k)|$
and denote by $T \subset \PP^N$ the closure of the image of $g$.

Assume that $\dim T = n-1$ and a general $g$-fibre contains a curve $C$ such that $H \cdot C = \frac{2}{n+1}$. Then we have
$$
2 \leq \deg T \leq \frac{k^{n-1} (n+1)}{2}
$$
and if $\deg T = \frac{k^{n-1} (n+1)}{2}$, we have the properties
from Corollary \ref{corollary-nminusone}.
\end{lemma}

\begin{corollary} \label{corollary-nminusone-terminal}
Let $X$ be a terminal degeneration of $\PP^n$.
Let 
$$
\merom{f}{X}{\PP^N}
$$
be the rational morphism defined by the linear system $|\sO_{X}(1)|$.
Denote by $Z \subset \PP^N$ the closure of the image of $f$, and let $d$ be its degree. 

If $\dim Z = n-1$, we have $2 \leq d < \frac{n+1}{2}$.
\end{corollary}

\begin{proof}
By Theorem \ref{theorem-case-nminusone} we already know that $d \leq \frac{n+1}{2}$, so we only have to exclude the case of equality. In this case by Corollary \ref{corollary-nminusone} the linear system
 $|\sO_X(1)|$ has no fixed component. Since $\rho(X)=1$, the fibration $f$ is not a morphism, so the birational morphism $\mu$ is not an isomorphism. Since $\mu$ is crepant
by Corollary \ref{corollary-nminusone}, this contradicts the terminality of $X$.
\end{proof}

In dimension two,  these general considerations already give the complete picture, which of course follows also from the main theorem in \cite{Man91}: 

\begin{corollary} \label{corollary-surfaces}
Let $X$ be a canonical degeneration of $\PP^2$. Then we have $X \simeq \PP^2$.
\end{corollary}

\begin{proof}
If $X \not\simeq \PP^2$, by Theorem \ref{theorem-euler-fields} the rational fibration $f$ has one-dimensional image $Z$. Thus
Theorem \ref{theorem-case-nminusone} applies, hence $2 \leq \deg Z \leq \frac{3}{2}$, a contradiction.
\end{proof}

\section{Stability} 

In general, a klt deformation $X$ of $\PP^n$ will not have a stable tangent bundle, as shown in Example \ref{ex:basic}, and thus  $X \ne \PP_n$.  In this section we show that instability is the only obstruction.

\begin{proof}[Proof of Theorem \ref{theorem-stability}]
Let us recall that since $\rho(X)=1$ for the klt degeneration $X$ of $\PP^n$ (Proposition \ref{proposition-picard-number}), so stability is always with respect to $H$.
Note also that for a torsion-free sheaf $\sF \subset X$, the intersection number
$c_1(\sF) \cdot H^{n-1}$ (and hence the slope $\mu_H(\sF)$) is well-defined even if $\det \sF$ is not $\Q$-Cartier.

Let $\merom{f}{X}{\PP^N}$ be the rational map defined by $|\sO_X(1)|$. If $f$ is generically finite onto its image, then $X \simeq \PP^n$ by Corollary \ref{cor:gen-finite}.
Thus we will assume that this is not the case and prove that $T_X$ is not semistable.

Let $H = M + B$ be the decomposition into mobile and fixed part of $|H|$.
Following Notation \ref{notation-Qfactorial},
let $\holom{\eta}{X'}{X}$ be a $\Q$-factorialisation
and $H' = M'+B'$ be the induced decomposition into mobile and fixed part. 
Let $\holom{\mu}{Y}{X'}$ be the modification associated to $M'$ given by  Lemma \ref{lemma-indeterminacies},  and let 
$$
\mu^* M'  \sim_\Q M_Y + \sum_{i=1}^k a_i E_i 
$$
be the decomposition provided by the lemma. The Stein factorisation of $\varphi_{|M|}$ defines a fibration 
$$
\holom{\varphi}{Y}{Z}
$$
such that $M_Y \simeq \varphi^* A$ for some ample globally generated Cartier divisor $A$, and by our assumption $m := \dim Z<\dim X$. 

Let now $\holom{\tau}{Z'}{Z}$ be a resolution of singularities,
and let $\holom{\sigma}{Y'}{Y}$ the induced map where $Y'$ is the normalisation of the dominant component of $Y \times_Z Z'$. Then we have an induced fibration $\holom{\varphi'}{Y'}{Z'}$
such that $$\sigma^* L \simeq \varphi'^* A'$$ where $A' := \tau^* A$ is a nef and big divisor.
We summarise the construction in a commutative diagram:
$$
\xymatrix{
Y' \ar[r]^{\sigma} \ar[d]^{\varphi'} & Y \ar[r]^\mu \ar[d]^\varphi & X' \ar[r]^{\eta} & X
\\
Z' \ar[r]^{\tau} & Z & &
}
$$

Consider next the cotangent sequence
$$
0 \rightarrow \varphi'^* \Omega_{Z'} \rightarrow \Omega_{Y'} \rightarrow \Omega_{Y'/Z'} \rightarrow 0 
$$
and take the saturation
$$
0 \rightarrow (\varphi'^* \Omega_{Z'})^{sat} \rightarrow \Omega_{Y'} \rightarrow \Omega_{\sF_Y} \rightarrow 0.
$$
Hence $c_1((\varphi'^* \Omega_{Z'})^{sat}) = \varphi^* K_{Z'} + G$ with $G$ an effective divisor.
Introducing the foliation 
$$ T_{\sF_Y}  := (\Omega_{\sF_Y})^* \subset T_{Y'},$$
we conclude 
$$
c_1(T_{\sF_Y}) = -K_{Y'} + \varphi^* K_{Z'} + G.
$$
The direct image $\eta_* \mu_* \sigma_* T_{\sF_Y} \subset T_{X}$ defines 
a foliation $T_{\sF} \subset T_{X}$ with
$$
c_1(T_{\sF}) = \eta_* \mu_* \sigma_* (-K_{Y'} + \varphi'^* K_{Z'} + G) \geq \eta_* \mu_* \sigma_* (-K_{Y'} + \varphi'^* K_{Z'}).
$$
Since $\eta_* \mu_* \sigma_* K_{Y'} = K_{X}$, we have
\begin{equation}
\label{c1-foliation}
c_1(T_{\sF}) \geq -K_{X} + \eta_* \mu_* \sigma_* \varphi'^* K_{Z'}.
\end{equation}

The quasi-polarised variety $(Z', A')$ is not birationally equivalent
to $(\PP^m, \sO_{\PP^m}(1))$ since 
$$
h^0(Z', \sO_{Z'}(A')) = h^0(Z, \sO_{Z}(A))= h^0(X, \sO_X(1)) \geq n+1 > m+1.
$$
Therefore it follows from Fujita's theory \cite{Fuj90}, \cite[Prop.2.15]{Hoe12} that
$K_{Z'} + m A'$ is pseudoeffective. Hence 
$$\eta_* \mu_* \sigma_* \varphi'^* (K_{Z'} + m A)$$ 
is pseudoeffective as well. 

Recalling that by construction we have $\eta_* \mu_* \sigma_* \varphi'^* A' = M$, it follows by
\eqref{c1-foliation} that 
$$
c_1(T_{\sF}) + m M   \geq -K_{X} + \eta_* \mu_* \sigma_* \varphi'^* (K_{Z'} + m A)
\geq -K_{X}.
$$
Thus we have 
$$
c_1(T_{\sF}) \geq -K_{X} - m M = (n-m+1) H + m B \geq (n-m+1) H.
$$
Thus the slope of the subsheaf $T_{\sF} \subset T_{X}$ with respect to $H$ is at least
$$
\frac{n-m+1}{n-m} H^n = 1 + \frac{1}{n-m}.
$$
But the slope of $T_{X}$ with respect to $H$ is $1 + \frac{1}{n}<1 + \frac{1}{n-m}$.
\end{proof} 

Next we present a related result via the so-called canonical extension. The ample class $c_1(\sO_{X}(1)) \in H^1(X, \Omega_X^{[1]})$ defines a non-split extension
$$ 
0 \to \sO_X \to \sE_X \to T_X \to 0,
$$
called the canonical extension; see \cite{GKP22}, \cite[Sect.4.B]{HP21} for details. 

\begin{theorem}\label{thm:stab2}
Let $X$ be a klt degeneration of $\PP^n$ that is smooth in codimension $2$.
If the canonical extension $\sE_{X}$ is semi-stable, then $X \simeq \mathbb P^n$.
\end{theorem}

\begin{proof} Since $\mathcal X_t \simeq \mathbb P^n$ for $t \ne 0$,
we have a Chern class equality
$$
{\frac{n}{2n+1}} c_1^2(\Omega^1_{\mathcal X_t})  \cdot A_t^{n-2} = c_2(\Omega^1_{\mathcal X_t}) \cdot A_t^{n-2}.
$$
where $A$ is any $\pi$-ample Cartier divisor on the total space $\mathcal X \rightarrow \Delta$.
Since $X$ is smooth in codimension two, the Chern classes
$c_1^2(\Omega^{[1]}_{X})$ and $c_2(\Omega^{[1]}_{X})$ are well-defined as linear forms, and the
equality remains true on $X$:
$$ 
\frac{n}{2n+1} c_1^2(\Omega^{[1]}_{X}) \cdot A_0^{n-2} = c_2(\Omega^{[1]}_{X}) \cdot A_0^{n-2}.
$$
Since $\mathcal E_X$ is semistable, by
\cite[Thm.1.3]{GKP22}  $X$ is a quasi-\'etale quotient of $\PP^n$ or an abelian variety. The latter case is excluded since $X$ is Fano. Thus we have $X \simeq \mathbb P^n/G$ with a finite group $G$ acting freely in codimension $2$.
Thus if  $\psi: \mathbb P^n \to \mathbb P^n/G = X$ is the quotient map, one has
$-K_{\mathbb P^n} = \psi^*(-K_{X_0})$.  Since
$$ 
(-K_{X})^n = (-K_{X_t})^n = (-K_{\mathbb P^n})^n, $$
the degree of $\psi$ is one. Hence $\psi$ is an isomorphism.
\end{proof}

\begin{remark*} {\rm It should be possible to prove Theorem \ref{thm:stab2} without assuming that $X$ is smooth in codimension $2$, using orbifold Chern classes.
Note also that neither the stability of $T_X$ implies the stability of the canonical extension and vice versa. 
}\end{remark*}

\section{Threefolds} 
\label{section-dim3}

In this section we will study canonical degenerations of $\PP^3$ and prove Theorem \ref{theorem-dim3}. By Example \ref{ex:basic} (in the case $n=3, d=2$)  there exists at least one non-trivial type of degenerations,
given by a cone $X$ over a normal quadric $Z \subset \PP^3$. 
This Fano threefold is the second item in Karzhemanov's classification \cite[Thm.1.6]{Kar15}
of Fano threefolds with canonical Gorenstein singularities and $(-K_X)^3=64$.

\subsection{First observations} \label{subsection-observations}

\begin{proposition} \label{proposition-dim3-gorenstein}
Let $X$ be a canonical degeneration of  $\PP^3$. Then $X$ is Gorenstein and $h^0(X, \sO_X(-K_X))=35$.
\end{proposition} 

\begin{remark} \label{rem: ADE}  
{\rm By \cite{Rei87} there is a finite set $N \subset X$ such that 
any $x \in X \setminus N$ has an open neighborhood $U$ such that $U \simeq \Delta \times S$
where $\Delta$ is a disc and $S$ a surface with ADE-singularities. In particular
$X \setminus N$ is $\Q$-factorial. However there are many canonical Gorenstein threefold
singularities which are not $\Q$-factorial, a possibility that will ultimately lead us to the third case in Theorem \ref{theorem-dim3}.
}
\end{remark}

\begin{proof} Let $\tau: V \to X$ be a $\Q$-factorial terminalisation. Then $V$ is a $\Q$-factorial weak Fano threefold with terminal singularities and $-K_V = \tau^*(-K_X)$. In particular by the Kawamata-Viehweg vanishing theorem 
$$\chi(V, \sO_V(-K_V))=h^0(V, \sO_V(-K_V)).$$

By \cite{Rei87} we know that for every Weil divisor $D$ on $V$ one has
$$
\chi(V, \sO_V(D))=
\frac{1}{12} D \cdot (D-K_V)(2D-K_V) +
\frac{1}{12} D \cdot c_2(V) + \sum_{P \in \mathbf B} c_P(D) + \chi(\sO_V)
$$
where $\mathbf B$ is the ``basket of singularities'' and $c_P(D)$ depends both on the singularities and the divisor $D$. In the case $D=-K_V$ one has (cf. \cite[Sect. (2.3)]{Pro07})
$$
c_P(-K_V) = \frac{r_P^2-1}{r_P} - \frac{b_P(r-b_P)}{r_P}
$$
where $\frac{1}{r_P}(1, b_P, -b_P)$ is the type of the singularity and $r_P$ is the Gorenstein index of the point $P$. As shown in \cite[Formula 2.6]{Pro07} this yields
\begin{equation}  \label {eq:1} 
h^0(V, \sO_V(-K_V)) -3 \leq - \frac{1}{2} K_V^3 - \frac{1}{2} \sum_{P \in \mathbf B} (1- \frac{1}{r_p}) \leq - \frac{1}{2} K_V^3. 
\end{equation}
In our case we have $(-K_V)^3 = (-K_X)^3 = (-K_{\PP^3}) = 64$ and by Lemma \ref{lemma-semicontinuity}
$$
h^0(V, \sO_V(-K_V)) = h^0(X, \sO_X(-K_X)) \geq h^0(\PP^3, \sO_{\PP^3}(-K_{\PP^3}))=35.
$$
Hence equality holds in \eqref{eq:1} and therefore $r_P=1$ for all $P \in \mathbf B$.
Hence $V$ is Gorenstein and so is $X$. 
\end{proof} 

Fano threefolds of high anticanonical degree have been studied intensively by the Russian school, in particular we know by \cite{PCS05}, \cite[Cor.3.4]{Kar15} that in the situation of Proposition \ref{proposition-dim3-gorenstein}
the anticanonical divisor $-K_X$ is very ample. Hence it defines an embedding
$$
\varphi_{|-K_X|}: X \hookrightarrow \PP^{34}
$$
as a variety of degree $64$. Note that the same holds for the anticanonical embedding of $\PP^3$, but it is not clear if these varieties correspond to points in the same connected component
of the Hilbert scheme of $\PP^{34}$.

\begin{remark} \label{remark-factorial}
In what follows we will use many times that a $\Q$-factorial Gorenstein threefold with terminal singularities is factorial \cite[Lemma 5.1]{Kaw88}, i.e., every Weil divisor is Cartier.
\end{remark}

\begin{lemma} \label{lemma-flops-nef}
Let $X$ be a Gorenstein Fano threefold with canonical singularities,
and let $M$ be a mobile effective Weil divisor on $X$ that is general in its linear system. Then there exists a 
$\Q$-factorial terminalisation
$$
\holom{\tau}{V}{X}
$$
such that the strict transform $M_V \subset V$ is nef.
\end{lemma}

\begin{proof} 
Let $\holom{\tau_0}{V_0}{X}$ be any $\Q$-factorial terminalisation.
Then $V_0$ is weak Fano, so its Mori cone is rational polyhedral.
Denote by $M_{V_0}$ the strict transform of $M$, then $M_{V_0}$ is mobile.
If $M_{V_0}$ is not nef there exists an $M_{V_0}$-negative extremal ray $\Gamma$ and let $\holom{\varphi_\Gamma}{V_0}{V'}$ denote the associated contraction. 
Since $M_{V_0}$ is mobile there are only countably many $M_{V_0}$-negative curves on $V_0$, so the contraction $\varphi_\Gamma$ is small. By Benveniste's theorem \cite[Thm.0]{Ben85} this implies that 
$$K_{V_0} \cdot \Gamma=0.$$
Since $-K_X$ is ample we note that the contraction is a contraction over $X$. Further, let $\psi_1: V_0 \dashrightarrow V_1$ be the associated flop. Then $V_1$ is still a  $\Q$-factorial terminal Gorenstein weak Fano threefold and
still a terminalisation of $X$. Moreover $M_{V_1} := (\tau_1)_* M_{V_0}$ is mobile, where $\tau_1: M_{V_1} \to X$ is the terminalisation. 
By the ``easy termination theorem'' \cite[Thm.6.43]{KM98}, the sequence of flops terminates with a model $\holom{\tau}{V}{X}$ that has the claimed properties.
\end{proof}

\begin{proposition} \label{proposition-not-big}
Let $X \not\simeq \PP^3$ be a canonical degeneration of $\PP^3$. Let
$$
|\sO_X(1)| = M + B
$$
be the decomposition into the mobile and  the fixed part.
Let $\holom{\tau}{V}{X}$ be a $\Q$-factorial terminalisation such that the strict transform $M_V \subset V$ is nef. Then $M_V$ is not big.
\end{proposition}

\begin{proof}
Since $\tau$ is crepant we have 
$$
-K_V \sim_\Q 4 \tau^* H \sim_\Q 4 (M_V+R) 
$$
where $R$ is an effective $\Q$-divisor that has no common components with $M_V$. If the linear system $|M_V|$ is not composed with a pencil, then a general element 
$S \in |M_V| $ is irreducible. If it is composed with a pencil, then $M_V \simeq m S$ with $S$ a prime divisor.
Since $S$ is a Cartier divisor (Remark \ref{remark-factorial}), by adjunction
$$
-K_S \sim_\Q (3 M_V + (M_V-S) + R)|_S.
$$
Note that $M_V-S$ is either zero or equal to $(m-1)S$,  hence in both cases the class
$(M_V-S)|_S$ is pseudoeffective.

Let $\holom{\nu}{S'}{S}$ be the composition of normalisation and minimal resolution, then
we have $K_{S'} \simeq \nu^* K_S - E$ with $E$ an effective divisor. Thus we obtain
\begin{equation}
\label{eqn-adjunction}
- K_{S'} \sim_\Q \nu^* (3 M_V + (M_V-S) + R) + E \geq 3 \nu^* M_V. 
\end{equation}
We now argue by contradiction and assume that $M_V$ is big.
Then the restriction of $M_V$ to the general divisor $S$ is a nef and big Cartier and so does the pull-back $\nu^* M_V$
is a nef and big Cartier divisor. By \cite[Prop.2.13c)]{Hoe12}, the inequality \eqref{eqn-adjunction} implies that there exists a birational morphism
$$
\holom{\psi}{S'}{\PP^2}
$$
such that $\sO_{S'}(\nu^* M_V) \simeq \psi^* \sO_{\PP^2}(1)$. Now observe that $K_{S'} \simeq \psi^* K_{\PP^2} + G$ with 
$G$ an effective divisor which is zero if and only if $\psi$ is an isomorphism. Yet by \eqref{eqn-adjunction} we have
$$
G \simeq K_{S'} - \psi^* K_{\PP^2} = -\nu^* (3 M_V + (M_V-S) + R) - E - \psi^* K_{\PP^2}
=  -\nu^*((M_V-S) + R) - E
$$
where for the last equality we used $-3 \nu^* M_V \simeq \psi^* K_{\PP^2}$.
Since the right hand side is anti-effective we obtain 
$$
G=0, 
\qquad
\mbox{and}
\qquad 
\nu^*((M_V-S) + R) + E=0.
$$
Note that $G=0$ and $E=0$ imply that $S' \simeq S \simeq \PP^2$. 
Furthermore $\nu^* (M_V-S)=0$ implies that the general element of $|M_V|$ is irreducible or $M_V^2=0$, the latter being a contradiction to $M_V$ being big. Now recall that
$$
\tau^* H \sim_\Q M_V + R = S + R.
$$
Since $\tau^* H$ is nef and big the effective $\Q$-divisor $S+R$ has connected support. Thus $\nu^* R=0$ implies that $R=0$.  This finally shows that
$$\tau^* H \sim_\Q S$$ is a prime Cartier divisor. Since $S$ is $\tau$-trivial
and $\tau$ is crepant we have $S \simeq \tau^* S_X$ with $S_X$ a {\em Cartier} divisor. In conclusion
$$
- K_X \simeq 4 H \sim_\Q 4 S_X.
$$
Since $S_X$ is Cartier we obtain $X \simeq \PP^3$ by the singular version of the  Kobayashi-Ochiai theorem \cite[Thm.1]{Fuj87},
 a contradiction. 
\end{proof}

Let $X$ be a canonical degeneration of $\PP^3$ such that $X \not\simeq \PP^3$.
We consider as in Section \ref{section-rational} the rational map 
$$
f: X \dasharrow Z
$$
given by the linear system $\vert \sO_X(1) \vert$. By Corollary \ref{cor:gen-finite} this
map is not generically finite onto its image, so $1 \leq \dim Z \leq 2$.
We start with the easier case $\dim Z=2$:

\begin{proposition} \label{proposition-dim-two}
Let $X$ be a canonical  degeneration of  $\PP^3$ such that $\dim Z=2$.
Then $X$ is the image of the contraction of the section 
$$
\PP(\sO_Q) \subset \PP(\sO_Q \oplus \sO_Q(-K_Q))
$$
where $Q \subset \PP^3$ is a normal quadric surface. 
\end{proposition}

\begin{proof}
By Theorem \ref{theorem-case-nminusone} we conclude $\dim Z = 2$, so $Z$ is a normal quadric surface. 
Then by Corollary \ref{corollary-nminusone} 
the linear system $|\sO_X(1)|$ has no fixed components, and the modification
$\holom{\mu}{Y}{X}$ given by Lemma \ref{lemma-indeterminacies} is crepant.
As in the lemma we write
$$
\mu^* H \sim_\Q M_Y + \sum_{i=1}^k a_i E_i,
$$
where $M_Y \simeq \varphi^* A$ is the pull-back of the hyperplane divisor
and $a_i>0$ for all $i$. Since $-K_X \simeq 4H$ is Cartier by Proposition \ref{proposition-dim3-gorenstein}, $a_i \in \frac{1}{4} \N$.
Since $\mu$ is crepant, 
\begin{equation}
\label{eq1}
- K_Y \simeq \mu^* (-K_X) \simeq \mu^* (4H) \sim_\Q 4 M_Y + \sum_{i=1}^k 4 a_i E_i.
\end{equation}

By Corollary \ref{corollary-nminusone} each exceptional divisor $E_i$ is contracted by $\mu$ onto a point and maps finitely onto the quadric surface $Z$.

{\em Step 1. We show that $$\sum_{i=1}^k a_i E_i= \frac{1}{2} E_1$$ and that $E_1 \cdot l=1$.}
Here $l$ denotes a general fiber of $\varphi: Y \to Z$ (recall that $\varphi$ has connected fibers). 
Since $\mu$ is crepant, the variety $Y$ has canonical singularities. Thus there exists a small
$\Q$-factorial modification, and we will replace $Y$ by this $\Q$-factorialisation (but only for this step of the proof). Observe first that since $\dim Z = 2$, 
the general fibre $l$ of the map $\varphi: Y \to Z$ meets the smooth locus of $Y$. In particular, $E_i \cdot l \geq 1$, since the divisors $E_i$ are Cartier near $l$. Thus
$$
2 = -K_Y \cdot l = \sum_{i=1}^k 4 a_i E_i \cdot l \geq \sum_{i=1}^k 4 a_i.
$$
However $a_i \geq \frac{1}{4}$ for every $i=1, \ldots, k$,
so we are done if we show that $a_1 > \frac{1}{4}$.
Arguing by contradiction assume that $a_1 = \frac{1}{4}$. Let 
$$
\holom{\nu}{S}{E_1}
$$
be the normalisation, then by subadjunction (Fact \ref{fact-subadjunction}) and \eqref{eq1} we have
$$
- K_S \geq - \nu^*(K_Y + E_1) = 4 \nu^* M_Y + \sum_{i \geq 2} 4 a_i \nu^* E_i
\geq 4 \nu^* M_Y.
$$
But $\nu^* M_Y$ is a nef and big Cartier divisor and $S$ is a normal surface, so $K_S+3 \nu^* M_Y$ is pseudoeffective by \cite[Prop.2.13.a)]{Hoe12}. This contradicts $-K_S \geq 4 \nu^* M_Y$
and finishes Step 1.

Since $\varphi|_{E_1}$ is finite and $E_1 \cdot l=1$ by Step 1, the divisor $E_1$ is a $\varphi$-section. Further, Equation \ref{eq1} now reads
\begin{equation} \label{eq1a} 
- K_Y  \sim_\Q 4 M_Y + 2E_1.
\end{equation}

{\em Step 2. We show that $\holom{\varphi}{Y}{Z}$ is a $\PP^1$-bundle in codimension one.}
Since $E_1$ is a section of $\varphi$, it is a normal surface. We claim that $E_1 \subset Y$ is a Cartier divisor over a big subset $Z^{\circ}$ of $Z$: otherwise by subadjunction that
$$
- K_Z \sim_\Q -K_{E_1} \sim_\Q -(K_{Y}+ E_1)|_{E_1} + \Delta
$$
with $\Delta$ a non-zero $\Q$-Weil divisor 
%\cite[Cor.16.7]{Uta92}. 
Since $\mu$ is crepant,
$-K_{Y}|_{E_1} \sim_\Q 0$. Therefore 
by Equation (\ref{eq1a}),   $- E_1|_{E_1} \sim_\Q 2 M_Y$, hence
$$
-(K_{Y} + E_1)|_{E_1} + \Delta \sim_\Q 2 M_Y + \Delta  > 2 M_Y = -K_Z. 
$$
Thus we have reached a contradiction and $E_1$ is Cartier over a big open subset.

By Equation (\ref{eq1a}) and since $-K_Y$ is $\varphi-$ample by Corollary \ref{corollary-nminusone}, the divisor $E_1$ is a $\varphi$-ample section that is Cartier over a big open subset. Therefore 
$\varphi$ has integral fibres in codimension one. The general fibre being $\PP^1$
this implies that $\varphi$ is a $\PP^1$-bundle in codimension one \cite[II,Thm.2.8]{Kol96}.

{\em Step 3. Conclusion.}
Since $\varphi$ is equidimensional the direct image $\varphi_* \sO_Y(E)$ is a reflexive sheaf of rank two \cite[Cor.1.7]{Har80}.  We have a canonical exact sequence
$$
0 \rightarrow \sO_Y \stackrel{s}{\rightarrow} \sO_Y(E_1) \rightarrow Q \rightarrow 0
$$
where $Q$ is a sheaf supported on $E_1$ which by Step 2 coincides with $\sO_{E_1}(E_1)$ in codimension one. Since $R^1 \varphi_* \sO_Y=0$ by relative Kodaira vanishing, the push-forward gives an extension 
$$
0 \rightarrow \sO_Z \rightarrow \varphi_* \sO_{Y}(E_1) \rightarrow T \rightarrow 0.
$$
The divisor $E_1$ being a section of $\varphi$, the section $s$ is non-zero on every fiber of $\varphi$. Thus the morphism $\sO_Z \rightarrow \varphi_* \sO_{Y}(E_1)$ is non-zero
at every point of $Z$, and consequently the sheaf $T$ is reflexive. 
Since $T = \varphi_* \sO_{E_1}(E_1)$ in codimension one  
and $-E_1|_{E_1} \sim_\Q 2 M_Y|_{E_1} \sim_\Q - K_Z$,
 the restriction of $c_1(T)$ to a general hyperplane section $C$
is isomorphic $K_Z|_C$. Yet $Z$ is a normal quadric surface, so the intersection numbers with
hyperplane sections determine the divisor class $c_1(T)$.
This shows that
$$ 
T \simeq K_Z.
$$ 
Since $\mbox{Ext}^1(K_Z, \sO_Z) \simeq H^1(Z, \sO_Z(-K_Z))=0$ by Kodaira vanishing,
we finally obtain that 
$$
\varphi_* \sO_{Y}(E_1) \simeq \sO_Z \oplus \sO_Z(K_Z).
$$
Thus $Y$ is isomorphic in codimension one to $Y':=\PP(\sO_Z \oplus \sO_Z(-K_Z))$.
Since $-K_Y$ and $-K_{Y'}$ are both relatively ample Cartier divisors over $Z$, we apply \cite[Lemma 6.39]{KM98} to see that $Y \simeq Y'$.
\end{proof}

\subsection{The case $\dim Z=1$: setup and birational modifications.}
\label{subsection-setup}
Let $X$ be a canonical  degeneration of $\PP^3$, and  - as before - with rational map
$$ f: X \dasharrow Z$$
given by the linear system $\vert \sO_X(1) \vert$. 
 We now assume $\dim Z = 1$, and let
$$
|\sO_X(1)| = M + B
$$
be the decomposition into fixed and mobile part. Let $\holom{\eta}{X'}{X}$ be 
a $\Q$-factorialisation, and let
$$
|\sO_{X'}(1)| = M' + B'
$$
be the induced decomposition into fixed and mobile part. Applying Lemma \ref{lemma-indeterminacies} to $M'$ we obtain a birational morphism
$$
\holom{\mu}{Y}{X}
$$
and a decomposition $\mu^* M' \sim_\Q M_Y + \sum a_i E_i$ such that 
$$
\holom{\varphi=\varphi_{|M_Y|}}{Y}{Z}
$$
resolves the indeterminacies of $\varphi_{|\sO_{X'}(M')|}$.

Since $X$ and hence $Y$ is rationally connected, we have
$Z \simeq \PP^1$ and therefore 
$$M_Y \simeq m F_Y$$ with $F_Y$ a general $\varphi$-fibre. 
By Lemma \ref{lemma-semicontinuity} $h^0(X, \sO_X(1)) \geq h^0(\PP^3, \sO_{\PP^3}(1))=4$, hence
$$
m+1 = h^0(\PP^1, \sO_{\PP^1}(m)) = h^0(Y, \sO_Y(M_Y))
=h^0(X, \sO_X(M)) \geq 4.
$$
Setting $F := \eta \circ \mu(F_Y)$, we can write
\begin{equation}
\label{M-divisible}
M \simeq m F
\end{equation}
with $m \geq 3$.

\begin{lemma} \label{lemma-fibre-space}
Let $X$ be a canonical  degeneration of $\PP^3$ such that $\dim Z = 1$.
Then there exists a $\Q$-factorialisation $\holom{\eta}{X'}{X}$ such that
the mobile part $M'$ is a nef divisor.
\end{lemma}

\begin{proof}
We start by considering an arbitrary $\Q$-factorialisation, then we will run a MMP to obtain $\eta$. 

{\em Step 1. Let $\holom{\eta_0}{X_0}{X}$ be any $\Q$-factorialisation,
and let $M_0$ be the mobile part of $|\sO_{X_0}(1)|$. Let $C \subset X_0$ be a curve such that $M_0 \cdot C<0$. We claim that $$-K_{X_0} \cdot C=0.$$ 
}

We argue by contradiction and suppose that $-K_{X_0} \cdot C>0$ (recall that $-K_{X_0}$ is nef). 
%By \eqref{M-divisible} we have $M \simeq m F$ with $m \geq 3$. 
Denote by $F_0 \subset X_0$ the strict transform, then we have $M_0 \simeq m F_0$ with $m\geq 3$ by (\eqref{M-divisible})
Since $F_0$ is mobile and $\dim X_0=3$, the curve $C$ is rigid, i.e., it does not deform in $X_0$.

Choose a $\Q$-factorial terminalisation $\holom{\tau_0}{V_0}{X_0}$, and let
$F_{V_0} \subset V_0$ be the strict transform of $F_0$. Then $F_{V_0}$ is Cartier 
by Remark \ref{remark-factorial} and mobile.

{\em 1st case. Assume that $C \not\subset X_{0,\sing}$.} Then the terminalisation $\tau_0$ is an isomorphism at the general point of $C$. Denote by $\tilde C$ the strict transform of $C$, then
$-K_{V_0} \cdot \tilde C = -K_{X_0} \cdot C>0$. The terminal Gorenstein threefold $V_0$ has hypersurface singularities \cite[Cor.5.38]{KM98}, 
so by \cite[Thm.II.1.15]{Kol96} the curve $\tilde C$ deforms in $V_0$. Since $\tilde C$ is not in the $\tau$-exceptional locus, the curve $C$ deforms in $X_0$, a contradiction.

{\em 2nd case. Assume that $C \subset X_{0,\sing}$.}
Since $\dim X_0=3$, the curve $C$ is an irreducible component of $X_{0,\sing}$. The terminal threefold $V_0$ has isolated singularities, 
so there exists a $\tau_0$-exceptional divisor $D_1$ such that $\tau(D_1)=C$. Since $C$ is in the base locus
of $|F_0|$ we may choose the exceptional divisor $D_1$ such that
$F_{V_0} \cap D_1$ surjects onto $C$. Thus if $f$ is an irreducible component of a general 
fibre of 
$$
\holom{\tau_0|_{D_1}}{D_1}{C},
$$
then $F_{V_0} \cdot f>0$. Since $F_{V_0}$ is Cartier we even have $F_{V_0} \cdot f \geq 1$. By Lemma \ref{lemma-estimate-negativity} below, $D_1 \cdot f \geq -2$. 

Let $D_1, \ldots, D_l$ be the $\tau_0$-exceptional divisors, then 
$$
\tau_0^* M_0 \sim_\Q \tau_0^* m F_0 \sim_\Q m F_{V_0} + \sum_{j=1}^l c_j D_j,
$$
with $c_j \geq 0$. By what precedes we have
$$
0 = \tau_0^* M_0 \cdot f =  m F_{V_0} \cdot f + c_1 D_1 \cdot f + \sum_{j \geq 2} c_j D_j \cdot f
\geq m - 2 c_1.
$$
Thus we obtain $c_1 \geq \frac{m}{2} \geq \frac{3}{2}$.

Since $\eta_0^* H \sim_\Q m F_0 + B_0$ we have a decomposition
$$
\tau_0^* \eta_0^* H \sim_\Q m F_{V_0} + c_1 D_1 + R
$$
where $R$ is an effective $\Q$-divisor that has no common components with $D_1$ or $F_{V_0}$.
The divisors $-K_{V_0}$ and $\tau_0^* \eta_0^* H$ being nef and $F_{V_0}$ being mobile,
this implies
$$
-K_{V_0} \cdot (\tau_0^* \eta_0^* H)^2 \geq 
-K_{V_0} \cdot (\tau_0^* \eta_0^* H) \cdot m F_{V_0} \geq 
-K_{V_0} \cdot c_1 D_1 \cdot m F_{V_0}. 
$$
By what precedes the intersection $F_{V_0} \cap D_1$ maps surjectively onto $C$, 
so $-K_{V_0} \cdot F_{V_0} \cdot D_1 \geq 1$. Since $m \geq 3$ 
and $c_1 \geq \frac{3}{2}$ we obtain that $-K_{V_0} \cdot c_1 D_1 \cdot m F_{V_0} \geq \frac{9}{2}$. Yet $\eta \circ \tau$ is crepant, so
$$
4 = -K_{X} \cdot H^2 = -K_{V_0} \cdot (\tau_0^* \eta_0^* H)^2.
$$
We have reached a contradiction
\footnote{\label{footnote-complicated}For later use observe that the argument above only uses that $C \subset \mbox{Bs}|F_0|$ is a curve such that $-K_{X_0} \cdot C>0$
and $C \subset X_{0,sing}$.},
this finishes the proof of Step 1.
 
{\em Step 2. Running the MMP.}
Let $\holom{\eta_0}{X_0}{X}$ be any $\Q$-factorialisation,
and let $M_0$ be the mobile part of $|\sO_{X_0}(1)|$. 
Choose $\epsilon>0$ such that the pair $(X_0, \epsilon M_0)$ is klt.
Recall that $X_0$ is weak Fano, so if $M_0$ is not nef, there exists
a $M_0$-negative curve $C \subset X_0$ that generates an extremal ray
in $\NE{X_0}$. By Step 1 we have $-K_{X_0} \cdot C=0$, in particular $C$ is $\eta_0$-exceptional.
Thus the contraction of the $K_{X_0}+\epsilon M_0$-negative extremal ray $\R^+ [C]$ factors through $\eta_0$. Moreover the ray is small, so we have a flip $X_0 \dashrightarrow X_1$
and a birational morphism $\holom{\eta_1}{X_1}{X}$.
Since $\eta_1$ is small, it is a $\Q$-factorialisation of $X$. 
Proceeding inductively we obtain a $K_{X_\bullet}+\epsilon M_{\bullet}$-MMP over $X$.
The MMP terminates by \cite[Thm.6.43]{KM98}, its outcome has the claimed properties.
\end{proof}

\begin{lemma} \label{lemma-estimate-negativity}
Let $X_0$ be a normal quasi-projective variety of dimension $n$ with canonical singularities, and let $\holom{\tau_0}{V_0}{X_0}$ be a $\Q$-factorial terminalisation. 
Let $E \subset V_0$ be a prime divisor such that $\tau(E) \subset X_0$ has codimension two.
Let $f \subset E$ be an irreducible component of a general fibre of $E \rightarrow \tau_0(E)$.
Then $E \cdot f \geq -2$. 
\end{lemma}

\begin{proof}
Let $S \subset X_0$ be a general complete intersection of $n-2$ hypersurfaces, and denote by $S_V \subset V_0$
its strict transform. Then $S$ has ADE singularities and 
$S_{V}$ is a terminal surface, hence smooth. Since $\tau_0$ is crepant, so is $\tau_0|_{S_{V}}$.
Thus $S_{V} \rightarrow S$ is the minimal resolution of an ADE-singularity and every exceptional curve has self-intersection $-2$.

Now note that
$E \cap S_{V}  = f + R$ where $R$ is an effective cycle  
on $S_{V}$ that does not contain $f$. Thus we have
$$
E \cdot f = (f+R) \cdot f \geq f^2=-2.
$$
\end{proof}

\begin{corollary} \label{corollary-fibre-space}
Let $X$ be a canonical  degeneration of $\PP^3$ such that $\dim Z = 1$.
Then there exists a $\Q$-factorialisation $\holom{\eta}{X'}{X}$ such that
$M' \simeq m F'$ with $|F'|$ a basepoint free pencil on $X'$.
\end{corollary}

\begin{proof}
By Lemma \ref{lemma-fibre-space} we can choose $\eta$ such that $M'$ and hence $F'$ is nef.
By construction $h^0(X', \sO_{X'}(F'))=2$, 
so $\mbox{Bs} |F'|$ is a complete intersection $D_1 \cap D_2$ with $D_i \in |F'|$ general elements.
Arguing by contradiction we have that $D_1 \cap D_2$ is not empty.

{\em Step 1. Reduction to terminalisation.}
Since $F'$ is nef, the 1-cycle $D_1 \cdot D_2$ is mobile and not contracted
by $\eta$. Thus we can find a curve $C \subset \mbox{Bs}|F'|$ such that $-K_{X'} \cdot C>0$.
By the proof of Lemma \ref{lemma-fibre-space} (cf. Footnote \ref{footnote-complicated}) we know that $C \not\subset X'_{\sing}$.

Let $\holom{\tau}{V}{X'}$ be a $\Q$-factorial terminalisation, and let $F_V \subset V$ be the strict transform of $F'$. Since $\tau$ is an isomorphism in the generic point of $C$, the base locus of $|F_V|$ is not empty. By Lemma \ref{lemma-flops-nef} we can assume (up to replacing $V$ by another terminalisation) that $M_V \simeq m F_V$ is nef. 
Since $\mbox{Bs}|F_V| \neq \emptyset$ we have $F_V^2 \neq 0$ and $M_V$ has numerical dimension at least two.

{\em Step 2. Assume that $M_V$ has numerical dimension two.} 
Since $M_V$ is nef and $V$ is weak Fano, the divisor $M_V$ is semiample, i.e. some multiple
$|k M_V|$ defines a fibration 
$$
\holom{\psi}{V}{T}
$$
onto a surface $T$.
Recall that $F_V$ is a Cartier divisor.
Thus $F_V$ being $\psi$-trivial and $-K_V$ being $\psi$-nef and big we know by the relative basepoint free theorem that  $F_V \simeq \psi^* F_T$  with $F_T$ an ample {\em Cartier} divisor on $T$.
We have 
\begin{equation}
\label{eq2} M_V^2 = m^2 \cdot F_T^2 \cdot [f]
\end{equation}
with $f$ a general $\psi$-fibre and
$$
\tau^* \eta^* H \sim_\Q M_V + R
$$
with $R$ an effective $\Q$-divisor that has no common component with $M_V$. 
Since $-K_V$ is nef, we deduce that
$$
-K_V \cdot (\tau^* \eta^* H)^2 \geq -K_V \cdot M_V^2 = -K_V \cdot (m^2 \cdot F_T^2 \cdot [f]).
$$
Since $-K_V \cdot f=2$ and $F_T^2 \geq 1$ for every ample Cartier divisor, we obtain that
 $-K_V \cdot (\tau^* \eta^* H)^2 \geq 18$.
 Yet $\eta \circ \tau$ is crepant,
 so
$$
4 = -K_X \cdot H^2 = -K_V \cdot (\tau^* \eta^* H)^2 
$$
gives the contradiction.

{\em Step 3. Conclusion.}
By Step 1 the divisor $M_V$ has numerical dimension at least two, and Step 2 excludes the possibility that it is equal to two. Thus $M_V$ is nef and big, in contradiction to Proposition \ref{proposition-not-big}.
\end{proof}

We can finally describe the geometry of $X$:

\begin{proposition} \label{proposition-classify-Xprime}
Let $X$ be a canonical  degeneration of $\PP^3$ such that $\dim Z = 1$.
Let $\holom{\eta}{X'}{X}$ be a $\Q$-factorialisation as in Corollary \ref{corollary-fibre-space}, and denote by
$$
\holom{\psi}{X'}{\PP^1}
$$
be the fibration defined by the linear system $|F'|$. Then the general fibre $F'$ is a del Pezzo surface
of degree four with canonical singularities such that 
$$
-K_{F'} \simeq 4 (B' \cap F').
$$
In particular the fixed part $B$ is not zero.
Moreover we have
$$
-K_X^2 \cdot F = 4, \qquad -K_X^2 \cdot B = 4, \qquad m=3
$$
\end{proposition}

\begin{proof}
We have $H \simeq m F + B$ and hence
$$
-K_{X'} \simeq 4 H' \simeq 4 m F' + 4 B'.
$$
By adjunction we obtain that $-K_{F'} \simeq 4 (B' \cap F')$. Since $X'$ only contains finitely
many $-K$-trivial curves and $F'$ is general, the divisor $-K_{F'}$ is ample, so $F'$ is a del Pezzo surface with canonical singularities. 
In particular we have
$(-K_{F'})^2 \leq 9$. Since $-K_{F'}$ is Cartier, we obtain that
$(-K_{F'})^2$ is either $4$ or $8$. Yet in the latter case
$$
16 = -K_{X'}^2 \cdot H' \geq
m (-K_{X'}^2) \cdot F' = 8 m \geq 24
$$ 
is a contradiction. 

In the case $(-K_{F'})^2=4$ the same computation and $-K_{X'}^2 \cdot B'>0$ yields
that $m=3$ and $-K_{X'}^2 \cdot B'=4$.
\end{proof}

We can actually be more precise:

\begin{proposition} \label{prop-identify-F-prime}
The surface $F'$ in Proposition \ref{proposition-classify-Xprime} has a unique singular point,
it is of type $D_5$.

If the singularity is of type $D_5$, we do not have 
$B=B_1+B_2$ with $2 B_2 \simeq F$.
\end{proposition}

The second statement is a bit obscure but will be needed later to exclude an annoying case.
The proof of this statement needs some preparation:

\begin{lemma} \label{lemma-differend}
Let $S$ be a del Pezzo surface of degree four with canonical singularities
such that $-K_S \simeq 4l$ with $l$ a curve.
Then the differend of $l$ has degree~$\frac{5}{4}$.
\end{lemma}

\begin{proof}
We have $-K_S \cdot l=1$ and $l^2=\frac{1}{4}$. Since $l$ is smooth (it is a line in the anticanonical embedding), we have
$K_l + \Delta \sim_\Q (K_S+l)|_l.$
Since $\deg K_l=-2$ the claim follows.
\end{proof}

The following result is an elementary exercise in birational geometry\footnote{Cf. \url{https://math.univ-cotedazur.fr/~hoering/students/fang.pdf}, Thm.3.2 for a detailed proof.}:

\begin{lemma} \label{lemma-add-Ak}
Let $S$ be a normal projective surface with canonical singularities, and let 
$\holom{\mu}{S}{S'}$ be an elementary birational Mori contraction with exceptional divisor $E$.
Then $S'$ is smooth in $\mu(E)$ and  there is at most one singular point of $S$ on $E$, this point being of type $A_k$. We have $(K_{S'})^2 = K_S^2+1+k$.
\end{lemma}

\begin{proposition} \label{degree-four-cases}
Let $S$ be a del Pezzo surface of degree four with canonical singularities
such that $-K_S \simeq 4l$ with $l \subset S$ a curve. Then there are three possibilities:
\begin{enumerate}
\item $\rho(S)=1$ and $S$ has a unique singular point, it is of type $D_5$.
\item $\rho(S)=1$ and $S$ has three singular points, two of type $A_1$ and one of type $A_3$.
\item $\rho(S)=2$ and $S$ has two singular points, one of type $A_1$ and one of type $A_3$.
\end{enumerate}
In the last two cases the curve $l$ passes through one $A_1$-singularity and the $A_3$-singularity, the differend of $l$ along the singularity has degree $\frac{1}{2}$ and $\frac{3}{4}$.
\end{proposition}

\begin{proof}
The minimal resolution is a weak del Pezzo surface of degree four, so it has Picard number six.
Thus $\rho(S)+k=6$ where $k$ is the number of exceptional divisors. 
With this information at hand we can settle the case $\rho(S)=1$ by looking at \cite[Table 1]{MZ88} and using that $\pic{S}/\mbox{Cl}(S)$ must contain an element of order four.

If $\rho(S)>1$, let $\holom{\psi}{S}{S'}$ be a birational Mori contraction. 
Then $-K_{S'} = 4 \psi_* l$ and $(-K_{S'})^2>(K_S)^2=4$. Since $(-K_{S'})^2 \leq 9$ is divisible by four, we obtain that $S'$ is a quadric surface, so it has at most an $A_1$-singularity. By Lemma \ref{lemma-add-Ak} we know that $S$ has a singularity of type $A_3$ along the exceptional divisor. Thus we are left to exclude the case where $S$ has a unique singularity of type $A_3$. Yet the degree of the differend of a smooth curve through an $A_3$-singularity is $\frac{3}{4}$ or $1$
\cite[16.6.2,16.6.3]{Uta92}, so we get a contradiction to Lemma \ref{lemma-differend}.
\end{proof}

\begin{proof}[Proof of Proposition \ref{prop-identify-F-prime}]
Let $\holom{\bar \psi}{\bar X}{\PP^1}$  be the relative anticanonical model of $X' \rightarrow \PP^1$. Since $X' \rightarrow \bar X$ only contracts $K$-trivial curves the rigidity lemma 
yields a small birational map $\bar \eta: \bar X \rightarrow X$.
Let $\bar H, \bar F, \bar B$ be the strict transforms of $H, F, B$,
then we have 
$$
-K_{\bar X} \simeq 4 \bar H  = 4 (3 \bar F + \bar B)
$$
Note that $-K_{\bar X} \simeq (\bar \eta)^* (-K_X)$ is Cartier, nef and big and $\bar \psi$-ample. Moreover $\bar B$ is $\Q$-Cartier since so are $\bar H$ and $\bar F$.

Since $X'$ only contains finitely many $K$-trivial curves we have $\bar F \simeq F'$, so it is enough to show the statement for $\bar F$. 

Let now $\bar B_1 \subset \bar B$ be the unique irreducible component that dominates $\PP^1$, i.e.
$\bar B_1 \cap \bar F=l$, and let $\holom{\nu}{\tilde B_1}{\bar B_1}$ be the normalisation.
Then $\nu^*(-K_{\bar X})$ is relatively ample and has degree one on the fibres, so
$\tilde B_1 \rightarrow \PP^1$ is isomorphic to a Hirzebruch surface $\mathbb F_d$. 
Set $e:=\nu^* (-K_{\bar X})^2$ and note that $1 \leq e \leq 4$ since
$(-K_{\bar X})^2 \cdot \bar B=4$ by Proposition \ref{proposition-classify-Xprime}.

For the first statement we will argue by contradiction and assume that we are in one of the two other cases of Proposition \ref{degree-four-cases}. 

{\em Step 1. The class $\bar B+\frac{3}{2} \bar F$ is pseudoeffective on its support and $\bar B+\frac{7}{4} \bar F$ is nef.}

It is enough to show that $\bar B+\frac{7}{4} \bar F$ is nef on its support. Since $\sO_{\bar X}(1)$ is nef, the statement is obvious for $F$ and the vertical components of $\bar B$, so we only
have to show that $\nu^*(\bar B+\frac{7}{4} \bar F)$ is nef. 
Analogously we see that is enough to show that $\nu^*(\bar B+\frac{3}{2} \bar F)$ is pseudoeffective.

Since $\bar X$ has canonical singularities it is $\Q$-factorial in codimension two \cite[Prop.9.1]{GKKP11}.
By Remark \ref{non-Q-factorial} the finite morphism $\nu$ induces a well-defined pull-back for $\Q$-Weil divisors classes and the subadjunction formula \eqref{eqn-subadjunction} holds. Thus we have
\begin{equation} \label{eqn4}
\nu^* (\bar B_1+\frac{3}{2}F)  \sim_\Q K_{\tilde B_1} + \Delta - \nu^* K_{\bar X} + \nu^* \frac{3}{2} F
\end{equation}
where $\Delta \geq 0$ is the differend.
We decompose
$\Delta=\Delta_h + \Delta_v$
into the horizontal (resp. vertical) part with respect to the ruling $\tilde B_1 \rightarrow \PP^1$. Denote by $C_0$ the negative section of $\tilde B_1 = \mathbb F_d$ and by $l$ a fibre of the ruling. Now observe the following
\begin{enumerate}
\item Since $\tilde B_1$ is a Hirzebruch surface, we have $\Delta_v \sim_\Q \alpha l$ with $\alpha \geq 0$ and $K_{\tilde B_1} \simeq -2C_0-(2+d)l$.
\item Since $-\nu^* K_{\bar X}$ is nef and has self-intersection $e$, we have
$-\nu^* K_{\bar X} \simeq C_0 + \frac{e+d}{2} l$ with $e \geq d$.
\item By Proposition \ref{degree-four-cases} we have $\Delta_h = \frac{1}{2} \Delta_{h,1}
+ \frac{3}{4} \Delta_{h,2}$ where the $\Delta_{h,i}$ are sections of the ruling. One of them
is not the negative section $C_0$ so it has class $C_0+\beta l$ with $\beta \geq d$.
Therefore we obtain $\Delta \sim_\Q \frac{5}{4} C_0 + \gamma l$ with $\gamma \geq \frac{d}{2}$.
\end{enumerate}
Combined with \eqref{eqn4} we obtain
\begin{equation} \label{differend-computation}
\nu^* (\bar B_1+\frac{3}{2}F) 
\sim_\Q \frac{1}{4} C_0 + (-\frac{1}{2}-\frac{d}{2}+\gamma+\alpha+\frac{e}{2}) l.
\end{equation}
Since $\alpha \geq 0$ and $\gamma \geq \frac{d}{2}$ 
the intersection number with $C_0$ is at least $\frac{e}{2}-\frac{d}{4}-\frac{1}{2}$.
Yet $e \geq d$, so this number is non-negative unless $d=e=1$. 
In the last case note that $\nu^* (\bar B_1+\frac{3}{2}F)  \sim_\Q \frac{1}{4} C_0 + \delta l$
with $\delta \geq 0$, so $\nu^* (\bar B_1+\frac{3}{2}F)$ is pseudoeffective.

Since $\nu^* (\bar B-\bar B_1)$ is vertical, hence a nef divisor, all the statements are still valid after adding $\bar B-\bar B_1$. The second statement follows by replacing $\frac{3}{2}$ by $\frac{7}{4}$ in the computation.

{\em Step 2. The contradiction.}
Using $\bar F^2=0$ we compute
\begin{multline*}
1 = \bar H^3 = 
[\frac{5}{4} \bar F + (\bar B + \frac{7}{4} \bar F)]
\cdot
[\frac{3}{2} \bar F + (\bar B+\frac{3}{2} \bar F)]^2
= \frac{17}{4} \bar F \cdot \bar B^2
+ (\bar B + \frac{7}{4} \bar F) \cdot  (\bar B+\frac{3}{2} \bar F)^2.
\end{multline*}
By Step 1 the class $\bar B + \frac{7}{4} \bar F$ is nef and $\bar B+\frac{3}{2} \bar F$
is nef in codimension one, so the last term is non-negative.
Yet $\bar F \cdot \bar B^2 = \frac{1}{4}$ by Proposition \ref{proposition-classify-Xprime}, so 
we get a contradiction.

{\em Step 3. Proof of the second statement.}
In Step 1 we used the assumption only in item c) to obtain the inequality $\gamma \geq \frac{d}{2}$. Since $\Delta_h$ is still a section of the ruling of $\tilde B_1$  we obtain that equality \eqref{differend-computation} holds, yet with the weaker property $\gamma \geq 0$.

Since $-K_X^2 \cdot B_2 = \frac{1}{2} (-K_X^2 \cdot F) = 2$
we have $-K_X^2 \cdot B_1=2$, so we have $e=2$.
Therefore 
\eqref{differend-computation} simplifies to
$$
\nu^* (B_1+\frac{3}{2}F) 
\sim_\Q \frac{1}{4} C_0 + (\frac{1}{2}-\frac{d}{2}+\gamma+\alpha) l.
$$
Since $2 B_2 \simeq F$ on $X$ we have $2 \bar B_2 \simeq F$
on $\bar X$ and therefore
$$
\nu^* (\bar B+\frac{3}{2}F)
=  \nu^* (\bar B_1+\bar B_2+\frac{3}{2}F)
\sim_\Q \frac{1}{4} C_0 + (1-\frac{d}{2}+\gamma+\alpha) l
$$
Since $d \leq e=2$ this divisor is pseudoeffective and
$\nu^* (\bar B+\frac{3}{2}F) \geq \frac{1}{4} C_0$.
If $\nu^* (\bar B+\frac{3}{2}F) \sim_\Q \frac{1}{4} C_0$
it is not always true that $\nu^* (\bar B+\frac{7}{4}F)$ is nef, but 
an elementary computation using that $C_0^2=-d \geq -e=-2$ shows that
$$
(\bar B + \frac{7}{4} \bar F) \cdot  (\bar B+\frac{3}{2} \bar F)^2 > \frac{-1}{16}.
$$
This is sufficient to get a contradiction as in Step 2.
\end{proof}

\subsection{An example} 

The Propositions \ref{proposition-classify-Xprime} and \ref{prop-identify-F-prime} give many restrictions on the Gorenstein Fano threefold $X$, we now show that such a variety actually exists.

\begin{example} \label{example-del-Pezzo-four}
Consider the vector bundle $V_d := \sO_{\PP^1}(2) \oplus \sO_{\PP^1}^{\oplus 2}$ over a curve $d \simeq \PP^1$. 
Set $\holom{\pi}{\PP(V_d)}{d}$, and let $\zeta_d$ be the tautological class.

Choose a splitting $\sO_{\PP^1}^{\oplus 2} \simeq L_1 \oplus L_2$ and set $L_0:= \sO_{\PP^1}(2)$. Let 
$$
s_0 \in H^0(\PP^1, L_0^{\otimes 2}), s_1 \in H^0(\PP^1, L_0 \otimes L_1),
s_2 \in H^0(\PP^1, L_0 \otimes L_2)
$$
be sections such that 
\begin{itemize}
\item $s_0 \neq 0$ has a simple root in $[1:0]$,
\item $s_2 \neq 0$ has a double root in $[1:0]$.
\item $s_1$ has at least\footnote{The section $s_1$ can be the zero section.} a double root in $[1:0]$
\end{itemize}
Set
$$
s_d = (s_0, s_1, 1, s_2, 0, 0) \in 
H^0(\PP^1, L_0^{\otimes 2} \oplus (L_0 \otimes L_1) \oplus L_1^{\otimes 2} \oplus (L_0 \otimes L_2)
\oplus (L_1 \otimes L_2) \oplus L_2^{\otimes 2}).
$$
Then $$s_d \in H^0(d, S^2 V_d) \simeq H^0(\PP(V_d), \sO_{\PP(V_d)}(2 \zeta_d))$$ defines a divisor $Y_d \subset \PP(V_d)$ such that
\begin{itemize}
\item the fibration $\holom{\varphi_d}{Y_d}{d}$ is a conic bundle having a double fibre over $[1:0]$ and all other fibres are smooth.
\item $Y_d$ is a normal projective surface with exactly one singular point;
\item the singular point is of type $D_4$;
\end{itemize}
Moreover $Y_d$ is a weak del Pezzo surface of degree four and Picard number two. The morphism
$$
\holom{\tau}{Y_d}{X_d}
$$
to the anticanonical model $X_d$ contracts a $\varphi_d$-section. The surface $X_d$ has a unique singular point, it is of type $D_5$.
\end{example}

\begin{proof}
Note first that since the section $1 \in H^0(\PP^1, L_1^{\otimes 2})$ does not vanish, the fibration $Y_d \rightarrow d$ is equidimensional. For every point
$p \in d$, the section $s_d$ determines the equation 
$$
s_0(p) X_0^2 + s_1(p) X_0 X_1 + X_1^2 + s_2(p) X_0 X_2 = 0
$$
of the fibre $\fibre{\varphi_d}{p} \subset \PP^2$. Since $L_0 \otimes L_2 \simeq \sO_{\PP^1}(2)$ and its $s_2$ has a double root in $[1:0]$, it does not vanish for $p \neq [1:0]$. An elementary computation now shows that the conic is smooth for $p \neq [1:0]$. For $p=[1:0]$
the equation simplifies to $X_1^2=0$, so the fibre is a double line. This shows the first item.

For the second item choose a local coordinate $t$ near $[1:0]$. Near this point the surface $Y_d$ is a hypersurface in $\Delta \times \PP^2$ given by the equation
$$
h:= s_0(t) X_0^2 + s_1(t) X_0 X_1 + X_1^2 + s_2(t) X_0 X_2 = 0.
$$ 
Since $s_1$ and $s_2$ have double roots at the origin and $s_0$ a simple root, a computation of the differential shows that 
the unique singular point of $Y_d$ is $(0, [0:0:1])$. 

The Picard number  $\rho(Y_d) = 2$, since the conic bundle $Y_d \rightarrow d$ has no reducible fibres. Since $K_{\PP(V_d)}=-3 \zeta_d$ and $Y_d \in |2 \zeta_d|$,
the anti-canonical bundle $-K_{Y_d} = \zeta_d$ is nef. Moreover $(-K_{Y_d})^2 =2 \zeta_d^3=4$, so $Y_d$ is weak del Pezzo of degree four. 
Since the minimal resolution has Picard number six, 
the singular point is of type $A_4$ or $D_4$.

Finally observe that $Y_d$ contains the section $X_0=X_1=0$, i.e. the section corresponding to the quotient $V \rightarrow V/(L_0 \oplus L_1) \simeq L_2$. Since $L_2$ is trivial, this
curve is contracted by the map to the anticanonical model $X_d$.
The surface $X_d$ satisfies the conditions of Proposition \ref{degree-four-cases} and has a unique singularity, so the singular point is of type $D_5$. Thus the singular point of $Y_d$ has type $D_4$.
\end{proof}

\begin{example} \label{the-example}
Let $T \simeq \mathbb F_4$ be the fourth Hirzebruch surface, denote by $C_0$ (resp. $d$) the negative section (resp. a fibre of the ruling). Consider the vector bundle
$$
V := \sO_T(12d+2C_0) \oplus \sO_T(6d) \oplus \sO_T.
$$
Set $\holom{\pi}{\PP(V)}{T}$, let $\zeta$ be the tautological class.

We fix $0 \neq s_{C_0} \in H^0(T, \sO_T(C_0))$. Since $\sO_T(12d+3C_0)$ is basepoint free
and $(12d+3C_0) \cdot C_0 = 0$ we can choose a section 
$$
\tilde s \in H^0(T, \sO_T(12d+3C_0))
$$
such that the corresponding divisor is a smooth curve that is disjoint from $C_0$. 
Choose any section $s_6 \in H^0(T, \sO_T(6d))$ and set
$$
s := (\tilde s \cdot s_{C_0},  s_6 \cdot s_{C_0}^2, 1, s_{C_0}^2, 0, 0) \in H^0(T, S^2 V \otimes  \sO_T(-12d)).
$$
Then $s$ defines a divisor $Y \subset \PP(V)$ such that
\begin{itemize}
\item the fibration $\holom{\varphi}{Y}{T}$ is a conic bundle having double fibres over $C_0$
and smooth fibres elsewhere;
\item $Y$ is a normal projective threefold with canonical singularities, the singular locus
is a section of $\fibre{\varphi}{C_0} \rightarrow C_0$
\item $Y$ has singularities of type $cD_4$ along this locus. 
\end{itemize}
Moreover $Y$ is a $\Q$-factorial weak Fano threefold with $(-K_Y)^3=64$ and Picard number three. The morphism
$$
\holom{\tau}{Y}{X}
$$
to the anticanonical model $X \subset \PP^{34}$ contracts exactly the $\varphi$-section $E_1$ corresponding to $V \rightarrow \sO_T$. We have
$$
- K_Y \simeq \varphi^* (12d+2C_0) + 2 E_1.
$$
and setting $B:=\frac{1}{2}\varphi^* C_0$, we have
$$
- K_{X} \simeq 4 (\tau_* (3 \varphi^* d+B)).
$$
The variety $X$ is a non-$\Q$-factorial Fano threefold with canonical singularities with Picard number one and degree 64. 
\end{example}

\begin{remark*}
Note that the divisor $B$ has integral coefficients, since the fibres over $C_0$ are double lines. Thus $-K_{X}$ is divisible by four as a Weil divisor.
\end{remark*}

\begin{proof}
Observe that for every fibre $d$ of the ruling, the restriction of $s$ to $d$
satisfies the conditions of Example \ref{example-del-Pezzo-four}:
the section $\tilde s \cdot s_{C_0}$ is not zero and vanishes with multiplicity
one in $d \cap C_0$, the section $s_6 \cdot s_{C_0}^2$ vanishes with multiplicity at least two in $d \cap C_0$,
and  $s_{C_0}^2$ with multiplicity exactly two in $d \cap C_0$. Thus the three first items follow from the corresponding items in Example \ref{example-del-Pezzo-four}. 

Since 
$$
K_{\PP(V)} \simeq \pi^* (K_T+\det V) - 3 \zeta
\simeq \pi^* 12d - 3 \zeta
$$
and $Y \in |2\zeta - \pi^* 12d|$ we see that $-K_Y \simeq \zeta|_Y$ is nef.
A standard computation shows that
$$
(-K_Y)^3 = \zeta^3 (2 \zeta - \pi^* 12d) = 64
$$
and
$$
h^0(Y, \sO_Y(-K_Y)) = h^0(T, V) = 35.
$$
Let us see that the anticanonical map contracts exactly the section $E_1$: since $-K_Y \simeq  \zeta|_Y$ it is equivalent to show that the $\zeta$-trivial curves $C \subset Y$
are contained in $E_1$. If $\varphi(C) \neq d$ this is clear, since then
the bundle
$$
\left( \sO_T(12d+2C_0) \oplus \sO_T(6l) \right) \otimes \sO_{\varphi(C)}
$$
is ample, so
there is only one $\zeta$-trivial curve in $\PP(V \otimes \sO_{\varphi(C)})$.
If $\varphi(C)=d$, the curve $C$ is contained in $Y_d \subset \PP(\sO(2) \oplus \sO^{\oplus 2})$.
Yet by Example \ref{example-del-Pezzo-four} this surface contains exactly
one $\zeta$-trivial curve, given by $E_1 \cap Y_d$.

Since $\tau$ is crepant and $Y$ has canonical singularities, so does $X$.
Since $\tau$ contracts exactly one prime divisor, the class group of $X$ has rank two.
The map 
$$
N_1(T) \simeq N_1(E_1) \rightarrow N_1(Y)
$$
is injective. Since $E_1$ is contracted by $\tau$ we obtain $\rho(Y/X)=\rk N_1(T)=2$. Thus
$X$ has Picard number one, and is not $\Q$-factorial.
\end{proof}

\subsection{The linear system $|2H|$} 

\begin{example} \label{example-Dfive}
Let $F$ be a del Pezzo surface of degree four with a $D_5$ singularity (the surface
no.24 in \cite[Table 3]{CP21}) which we consider in its anticanonical embedding
$F \subset \PP^4$. The surface $F$ contains a unique line $l$ and $\mbox{Cl}(X) \simeq \Z l$.
Since $(K_F)^2=4$ we see that 
$$
-K_F \simeq 4 l
$$
and $l^2=\frac{1}{4}$.

Moreover $F$ contains a unique pencil $|2l|$ of conics, the unique basepoint of this pencil is the singular point of $F$. Let $\holom{\nu}{\tilde F}{F}$ be the minimal resolution, and denote by $\tilde C \subset \tilde F$ the strict transform of a general conic $C \in |2l|$.
Note that $-K_{\tilde F} \cdot \tilde C=-K_F \cdot C=2$ implies that $C^2=0$, so the
linear system $|\tilde C|$ is a basepoint free pencil on $\tilde F$.
An elementary computation\footnote{It is not difficult to figure out the sequence of five blowups given by the MMP $\tilde F \rightarrow \PP^2$, this determines the configuration of the curves $E_1,\ldots,E_5, \tilde l, \tilde C$.} shows that
\begin{equation}
\label{pullbackconic}
\nu^* C = \tilde C + E_1+E_2+E_3 + \frac{1}{2} (E_4+E_5)
\end{equation}
where the $E_i$ are the exceptional curves and the curves $E_4$ and $E_5$ corresponding to the short edges of the $D_5$-graph. Moreover $\tilde C \cdot E_1=1$ and $\tilde C$ is disjoint from the other exceptional curves.

Let $\tilde l \subset \tilde F$ be the strict transform of the line $l$. Then, up to renumbering $E_4$ and $E_5$, we have $\tilde l \cdot E_5=1$
and $\tilde l$ is disjoint from the other exceptional curves.

Let $\holom{\mu_F}{F'}{F}$ be the extraction of the divisor $E_1$, i.e.
$F'$ is the surface obtained by contracting the curves $E_2+\ldots+E_5$. 
Denote by $C' \subset C$, then \eqref{pullbackconic} implies that
$$
\mu_F^* C = C' + E_1
$$
Since $\tilde C$ is disjoint from the exceptional locus $E_2+\ldots+E_5$ we have $(C')^2=(\tilde C)^2=0$, so the pencil $|C'|$ is basepoint free and defines a fibration
$$
\holom{\psi}{F'}{d \simeq \PP^1}.
$$
Looking at the finite map $\holom{\mu_F \times \psi}{F'}{F \times \PP^1}$
we see that $F'$ is the graph of the rational map $\varphi_{|2l|}=\varphi_{|C|}$.

Since $F$ contains a unique line, the conic bundle $\psi$ has a unique singular fibre
given by the strict transform of the double $2l$. By construction the surface $F'$ has a unique singular point, this point is of type $D_4$ and is the intersection of the double line and $E_1$.
\end{example}

We fix the setup for the rest of the section

\begin{setup} \label{setup-dim3}
{\rm Let $X$ be a canonical degeneration of $\PP^3$ such that $\dim Z=1$, by Proposition
\ref{proposition-classify-Xprime} we have a decomposition
$$
H \simeq 3 F + B
$$
into mobile and fixed part. Moreover we can find a $\Q$-factorialisation $\holom{\eta}{X'}{X}$ such that the strict transform $F'$ is a basepoint-free pencil defining a fibration
$$
\holom{\psi}{X'}{\PP^1}.
$$
By Proposition \ref{prop-identify-F-prime} the general fibre $F'$ 
is a del Pezzo surface of degree four with a $D_5$ singularity. 
We decompose
\begin{equation}
\label{decomposeH}
H' \simeq 3 F' + B_h + B_v
\end{equation}
where $B_h+B_v$ is the decomposition of the fixed part $B'$ into the horizontal and vertical part.

Note that from now 
$$
g: X \dashrightarrow T
$$ 
will denote the rational map $\varphi_{|2H|}$ defined by the linear system $|2H|$.

}
\end{setup}

\begin{proposition} \label{proposition-geometry-2D}
In the situation of the Setup \ref{setup-dim3}, the linear system $|2H|$ 
defines a rational map
$$
g: X \dashrightarrow T
$$  
onto a rational normal scroll $T \subset \PP^9$ of degree eight. The closure of a general fibre
identifies to a conic in the del Pezzo surface $F$ and the image of $F$ in $T$ is a line of the ruling.
\end{proposition}

\begin{proof}
Note first that $(2H)^3=8$ and by Lemma \ref{lemma-semicontinuity}
$$
h^0(X, \sO_X(2H)) \geq h^0(\PP^3, \sO_{\PP^3}(2))=10.
$$
It is equivalent to show the statement for the fibration 
$\merom{g}{X'}{T}$ defined by $|2H'|$ on the $\Q$-factorialisation $X'$. 

{\em Step 1. We show that a connected component of a general $g$-fibre identifies to a conic in the del Pezzo surface.}
By \eqref{decomposeH} we have $2H' = 6F' + 2B_v + 2 B_h$
and we claim that that the restriction morphism
\begin{equation}
\label{restriction}
H^0(X', \sO_{X'}(2H')) \rightarrow H^0(F, \sO_F(2 B_h))
\end{equation}
has rank at least two: otherwise $2B_h$ is in the fixed part of $|\sO_{X'}(2H')|$ 
and therefore
$$
H^0(X', \sO_{X'}(2H'))  \simeq H^0(X', \sO_{X'}(6F'+2B_v)).
$$
We claim that $\psi_* \sO_{X'}(2B_v) \simeq \sO_{\PP^1}(a)$ with $a \leq 1$. Note that this implies that
$$
H^0(X', \sO_{X'}(2H')) = H^0(\PP^1, \psi_* \sO_{X'}(6F'+2B_v)) \simeq \sO_{\PP^1}(a+6)) 
$$
has dimension at most eight, a contradiction.

{\em Proof of the claim.} Since $\sO_F(B_v) \simeq \sO_F$ the direct image has rank one, so the statement is equivalent to $h^0(X', \sO_{X'}(2B_v)) \leq 2$. We decompose 
$$
2 B_v \simeq Mob + Fix.
$$
into its mobile and fixed part. The mobile divisor
$Mob$ is disjoint from the general fibre $F$, so $Mob \simeq \lambda F'$ for some 
$\lambda \in \N$ and the claim is equivalent to showing $\lambda \leq 1$.
Since $B_h \neq$ we have $(-K_{X'})^2 \cdot B_v < 4$ by Proposition \ref{proposition-classify-Xprime}.
Therefore  we have
$$
(-K_{X'})^2 \cdot Mob \leq (-K_X)^2 \cdot 2B_v \leq 6.
$$
Since $(-K_X)^2 \cdot F'=4$, this ends the proof of the claim.

By Example \ref{example-Dfive} we have $h^0(F, \sO_F(2 B_h))=h^0(F, \sO_F(2 l))=2$. We have shown
that the restriction map
$$
H^0(X', \sO_{X'}(2)) \rightarrow H^0(F, \sO_F(2 B_h)) \simeq H^0(F, \sO_F(2 l)) 
$$
has rank two, so it is surjective. Thus
the restriction of
$\varphi_{|2H'|}$ to $F$ coincides with the rational map 
$$
\varphi_{|2 l|} : F \dashrightarrow \PP^1
$$
The general fibre of $\varphi_{|2 l|}$ is a conic $C$ in the del Pezzo surface, so this finishes the proof of Step 1. Note that by Example \ref{example-Dfive} 
the graph of $\varphi_{|2 l|}$ is a conic bundle $F' \rightarrow \PP^1$ having exactly one singular fibre, a double line. Along this double line we have a unique singular point,  it is of type $D_4$.

{\em Step 2. We have $h^0(X, \sO_X(2H))=10$ and $\deg T=8$.}

Note first that 
the surface $T \subset \PP^N$ is linearly non-degenerate,
since $g^* |\sO_{\PP^N}(1)| = |\sO_X(2)|$.
Thus we have by a theorem of Del Pezzo \cite{EH87} 
$$
8 \geq \deg T \geq \codim_{\PP^N} T + 1 \geq 7+1
$$
where in the last step we used $N=h^0(X, \sO_X(2H))-1 \geq 9$. 
By Step 1 the rational map contracts curves $C$ such that
$H \cdot C=\frac{1}{2}$, so Lemma \ref{lemma-variant} yields
$\deg T \leq 8$. Thus we have equality.

Note that by the case of equality in Lemma \ref{lemma-variant} 
we now know that $|2H|$ has no fixed component, so we can apply 
Lemma \ref{lemma-indeterminacies} to $2H$ to get a birational 
map $\holom{\mu}{Y}{X}$ and a fibration $\holom{\varphi}{Y}{T}$
with the properties given by the lemma. 

{\em Step 3. $T$ is a rational normal scroll.}
By Step 2 the surface $T \subset \PP^9$ has minimal degree, so by \cite{EH87} it is a rational normal scroll or a cone over the rational normal curve of degree eight in $\PP^8$. 

Arguing by contradiction assume that $T$ is a cone. Then 
$c_1(\sO_T(1))=8 l$ with $l$ a line of the ruling. Since 
$\mu^* 2H \simeq \varphi^* c_1(\sO_T(1)) + E$ with $E$ a $\mu$-exceptional divisor. We obtain that $H \simeq \mu_* \varphi^* 4 l$.
Yet $\mu_* \varphi^* 4 l$ has no fixed component, a contradiction to 
$B \neq 0$.

{\em Step 4. A general surface $F$ is mapped by $g$ onto a line of the ruling in $T$.}

We know that $F$ is contracted onto a curve $G \subset T$ since $F$ is covered by conics that are contracted by $g$. Thus we have 
$F \simeq \mu_* \varphi^* G$ and $h^0(X, \sO_X(F))=2$ implies that
$h^0(T, \sO_T(G)) \leq 2$. Yet a general curve $G \subset T$ is 
a nef divisor and an easy computation on Hirzebruch surfaces shows that
$h^0(T, \sO_T(G)) \leq 2$ implies that $G$ is a line of the ruling.
\end{proof}

\begin{lemma} \label{lemma-description-Y}
In the situation of Proposition \ref{proposition-geometry-2D},
let $\mu: Y \rightarrow X$ and $\varphi: Y \rightarrow T$
be the maps given by the case of equality in Lemma \ref{lemma-variant}.
Then the $\mu$-exceptional locus is a unique prime divisor $E_1$ and the fibration $\holom{\varphi}{Y}{T}$ is a conic bundle.

Denote by $F_Y \subset Y$ (resp. $B_Y \subset Y$) the strict transform of $F \subset X$ (resp. $B \subset X$). Then we have
$$
T \simeq \mathbb F_4,
$$
the discriminant divisor of $\varphi$ is the negative section $C_0 \subset \mathbb F_4$ and $2 B_Y \simeq \varphi^* C_0$.
\end{lemma}

\begin{proof}
By Lemma \ref{lemma-variant} the fibration $\varphi$ is an (abstract) conic bundle with general fibre $C$.
Since 
$$
-K_X \cdot \mu(C) = -K_Y \cdot C = 2
$$
and $-K_X$ is very ample (cf. Subsection \ref{subsection-observations}), we can even identify
the $\varphi$-fibres to conics in the anticanonical embedding $X \hookrightarrow \PP^{34}$.
 
This implies that there is just one $\mu$-exceptional divisor:
otherwise the family of conics $\mu(C)$ would have at least two fixed points, but we know 
by Example \ref{example-Dfive} that
the family of conics on the del Pezzo surface $F \subset X$ has only one fixed point.
Note also that $E_1$ is a section since otherwise the irreducible conics $\mu(C)$ would be singular in the point $\mu(E_1)$.

Since $2H \cdot C=1$, we obtain
\begin{equation}
\label{eqn5}
\mu^* 2H \sim_\Q \varphi^* \sO_T(1) + E_1
\end{equation}
Let now $B' \simeq B_h+B_v$ be the decomposition of the fixed part of $|H'|$ from \eqref{decomposeH},
and denote by $B_Y \simeq B_{h,Y} + B_{v,Y}$ the induced decomposition of the strict transform.

We already know by Proposition \ref{proposition-geometry-2D} that
$F_Y \simeq \varphi^* d$ with $d$ a line of the ruling.
The divisor $B_{h,Y}$ is also vertical with respect to $\varphi$:
for $F$ a general $\psi$-fibre
the conic bundle contracts the line $B_h \cap F$ since they appear as double lines in the family of conics. Thus $B_{h,Y}$ is contracted onto a curve $C_h \subset T$
and $\varphi^* C_h = 2 B_{h,Y}$ since the fibres over $C_h$ are double lines.
In summary
$$
\mu^* 2H \sim_\Q 6 F_Y + 2 B_{h,Y} + 2 B_{v,Y} + c_1 E_1 \simeq \varphi^* (6d+C_h) + 2 B_{v,Y} + c_1 E_1.
$$
Since $B_{v,Y}$ is disjoint from a general conic, we have $B_{v,Y} \cdot l=0$. Thus 
$E_1 \cdot l=1$ and $2H \cdot \mu(C)=1$ yields $c_1=1$.
Combined with \eqref{eqn5} we obtain
$$
\varphi^* \sO_T(1) = \varphi^* (6d+C_h) + 2 B_{v,Y},
$$
in particular $2 B_{v,Y} \sim_\Q \varphi^* \Delta$ for some effective $\Q$-divisor $\Delta$ on $T$. Since $\varphi$ is a conic bundle, a pull-back of an irreducible curve is 
either reduced or has multiplicity two. Thus $\Delta$ has integral coefficients. Observe that
$$
8 = c_1(\sO_T(1))^2 = c_1(\sO_T(1)) \cdot (6d+C_h+\Delta) \geq 6 + c_1(\sO_T(1)) \cdot C_h
$$
and the inequality is strict if and only if $\Delta \neq 0$. In particular $c_1(\sO_T(1)) \cdot C_h \leq 2$.

Now recall that the rational normal scrolls of degree eight in $\PP^9$ are isomorphic to Hirzebruch surfaces
$\mathbb F_e$ embedded by the linear system $|C_0+(4+\frac{e}{2})d|$ with $C_0$ the negative section. Such an embedding exists if and only if $e$ is even and $e \leq 6$. Observe that for $e \leq 2$ we have
$$
c_1(\sO_T(1)) \cdot C = (C_0+(4+\frac{e}{2})d) \cdot C_h \geq 3
$$ 
unless $C_h$ is a line. But this is impossible since $C_h$ is a section of the ruling. 
Since $c_1(\sO_T(1)) \cdot C_h \leq 2$ we obtain $e=4$ or $e=6$.

{\em 1st case. We have $e=4$.}
Then 
$$
2 \geq c_1(\sO_T(1)) \cdot C_h = (C_0+6d) \cdot C_h \geq 2
$$
implies that $C_h=C_0$ and $\Delta=0$.

{\em 2nd case. We have $e=6$.}
Then
$$
2 \geq c_1(\sO_T(1)) \cdot C_h = (C_0+7d) \cdot C \geq 1
$$
implies that $C_h=C_0$ and $c_1(\sO_T(1)) \cdot \Delta=1$. Since $\Delta$ is distinct from $C_h$,
it is a line $d_0$ of the ruling. Since $2 B_{v,Y} \simeq \varphi^* d_0$, the $\varphi$-fibres over $d_0$ are non-reduced.
In particular we have $B=B_h + B_v$
and $2 B_v = \mu_* \varphi^* d_0 \simeq \mu_* \varphi^* d \simeq F$.
Yet this is excluded by the second statement of Proposition \ref{prop-identify-F-prime}, so the case $e=6$ does not exist.
\end{proof}

The next step is to determine the embedding 
$$
i: Y \hookrightarrow \PP(\varphi_* \sO_Y(-K_Y))
$$
that gives the conic bundle structure. The direct image $V:=\varphi_* \sO_Y(-K_Y)$ is a rank three vector bundle, and we have an extension
\begin{equation}
\label{extensionV}
0 \rightarrow W \rightarrow V \rightarrow \sO_T \rightarrow 0
\end{equation}
given by the section $E_1 \subset Y$. Thus we are left to determine the rank 2 vector bundle $W$.

\begin{proposition} \label{proposition-extension-W}
In the situation of Proposition \ref{proposition-geometry-2D} we
have an extension
\begin{equation}
\label{extensionW}
0 \rightarrow \sO_T(ad+2C_0) \rightarrow W \rightarrow \sO_T(bd) \rightarrow 0
\end{equation}
with $a,b \in \Z$.
\end{proposition}

\begin{proof}
We claim
$$
V \otimes \sO_d \simeq \sO_{\PP^1}(2) \oplus \sO_{\PP^1}^{\oplus 2}, \qquad W \otimes \sO_d \simeq \sO_{\PP^1}(2) \oplus \sO_{\PP^1}
$$
for {\em every} fibre $d$ of the ruling $h: \mathbb F_4 \rightarrow \PP^1$.
The claim immediately implies the statement:
 the direct image $h_* (W \otimes \sO(-2C_0))$ has rank one, and the natural morphism
$$
h^* h_* (W \otimes \sO(-2C_0)) \rightarrow W \otimes \sO(-2C_0)
$$
defines a rank one subbundle isomorphic to $\sO_T(ad)$ for some $a \in \Z$.

{\em Proof of the claim.}
We already know that the statement holds for the general fibre by Example \ref{example-Dfive}.
Thus we have $c_1(V \otimes \sO_d)=2$. Moreover since $\varphi$ is equidimensional, the surface
$$
Y_d:= \fibre{\varphi}{d} \subset \PP(V \otimes \sO_d)
$$
has class $2 \zeta$ and 
$$
-K_Y|_{Y_d} \simeq -K_{Y_d} \simeq \zeta_d.
$$
Since $-K_Y|_{Y_d}$ is nef we obtain that $\zeta_d$ and therefore $V \otimes \sO_d$ is nef.
Since the splitting type on the general fibre is $\sO_{\PP^1}(2) \oplus \sO_{\PP^1}^{\oplus 2}$,
semicontinuity implies that $V \otimes \sO_d$ has the same splitting type.
\end{proof}

\begin{theorem}
In the situation of Setup \ref{setup-dim3}, the variety $X$ is isomorphic to the anticanonical model from Example \ref{the-example}.
\end{theorem}

\begin{proof}
Equivalently we show that $Y$ is isomorphic to the variety of the same name in Example \ref{the-example}. By Lemma \ref{lemma-description-Y}, $T \simeq \mathbb F_4$ and the discriminant divisor of $\varphi$
coincides with the negative section $C_0$.
Note also that as a consequence of \eqref{eqn5} we have
\begin{equation}
\label{equationE1}
E_1|_{E_1} \sim_\Q c_1(\sO_T(-1)) \simeq -(6d+C_0).
\end{equation}

Let $C \subset \mathbb F_4$ be a general curve in the linear system $|4d+C_0|$. The curve $C$ is disjoint
from the negative section, so the surface $Y_C :=\fibre{\varphi}{C} \rightarrow C$ is a ruled surface. The curve
$E_1 \cap Y_C$ is a section and \eqref{equationE1} yields
$$
(E_1 \cap Y_C)^2 = E_1^2 \cap \varphi^* C = -6.
$$
Thus we have $Y_C \simeq \mathbb F_6$. Using
\eqref{eqn5} and \eqref{equationE1} we obtain
$$
(-K_Y|_{Y_C})^2 = (4 \mu^* H)^2 \cdot \varphi^* (4d+C_0) = 24.
$$
Since $-K_Y$ has degree two on the fibres of $Y_C \rightarrow C$,
this determines the class of the invertible sheaf $\sO_{Y_C}(-K_Y)$ on $\mathbb F_6$
and we deduce that
$$
18 = \deg (\varphi_*  \sO_{Y_C}(-K_Y))
= \deg((\varphi_* \sO_Y(-K_Y)) \otimes \sO_C).
$$
Therefore we have $a+b=18$ in \eqref{extensionW}. 
Since 
$$
h^0(T, V) = h^0(Y, \sO_Y(-K_Y)) = 35, 
$$
we get $h^0(T, W) \geq 34$ by \eqref{extensionV}. An elementary computation using $a+b=18$ and \eqref{extensionW}
shows that $h^0(T, W)<34$ if $a<8$. Thus we have $a \geq 8$ and therefore
$H^1(T, W)=0$. Thus we have $h^0(T, W)=h^0(T, V)-1=34$ and finally obtain $a=12$.

Thus we have an extension
$$
0 \rightarrow \sO_T(12d+2C_0) \rightarrow W \rightarrow \sO_T(6d) \rightarrow 0.
$$
The extension class is in $H^1(T,  \sO_T(6d+2C_0)) \simeq \C$ and it is not difficult to see that we have an injection
$$
H^1(T,  \sO_T(6d+2C_0)) \hookrightarrow H^1(C_0, \sO_{C_0}(6d+2C_0)) = H^1(\PP^1, \sO_{\PP^1}(-2)).
$$ 
Thus if the extension is not split, we have
$W \otimes \sO_{C_0} \simeq \sO_{\PP^1}(5)^{\oplus 2}$ and hence
$$
V \otimes \sO_{C_0} \simeq \sO_{\PP^1}(5)^{\oplus 2} \oplus \sO_{\PP^1}.
$$
Yet $\PP(V \otimes \sO_{C_0})$ contains the divisor $B_Y$ which is isomorphic to $\mathbb F_4$. Thus we have a quotient $\sO_{\PP^1}(5)^{\oplus 2} \oplus \sO_{\PP^1} \twoheadrightarrow \sO_{\PP^1}(4) \oplus \sO_{\PP^1}$, a contradiction.

We know that $Y \subset \PP(V)$ has class $2 \zeta + \varphi^* (md)$ since
$Y_d \subset \PP(V \otimes \sO_d)$ has class $2 \zeta_d$ by Example \ref{example-del-Pezzo-four}.
Since 
$K_{Y_C}^2=8$ 
and 
$$
K_{Y_C}^2 = (K_{\PP(V)}+Y+ \varphi^* C)^2 \cdot Y \cdot \varphi^* C
$$
a straightforward computation shows that  
$$
[Y] = 2 \zeta + \varphi^* (-12d).
$$
\end{proof}

\providecommand{\bysame}{\leavevmode\hbox to3em{\hrulefill}\thinspace}
\providecommand{\MR}{\relax\ifhmode\unskip\space\fi MR }
\providecommand{\MRhref}[2]{%
  \href{http://www.ams.org/mathscinet-getitem?mr=#1}{#2}
}
\providecommand{\href}[2]{#2}

%\bibliographystyle{amsalpha}
%\bibliography{biblio}

\end{document}